\documentclass[a4paper,10pt]{elsarticle}
\usepackage[utf8x]{inputenc}

\usepackage{amsmath,amssymb,amsthm}
\usepackage{enumerate}
\usepackage{fca}

\theoremstyle{definition}
\newtheorem{defn}{Definition}[section]

\newtheorem{remark}[defn]{Remark}

\theoremstyle{plain}
\newtheorem{theorem}[defn]{Theorem}
\newtheorem{proposition}[defn]{Proposition}
\newtheorem{corollary}[defn]{Corollary}
\newtheorem{lemma}[defn]{Lemma}
\newtheorem{construction}[defn]{Construction}
\newtheorem{conjecture}[defn]{Conjecture}

 \DeclareMathOperator{\Aut}{Aut}
\DeclareMathOperator{\CoAt}{CoAt} 
\DeclareMathOperator{\End}{End} \DeclareMathOperator{\id}{id}
 \DeclareMathOperator{\JM}{JM}
\DeclareMathOperator{\Mat}{Mat} \DeclareMathOperator{\Min}{Min}
\DeclareMathOperator{\Res}{Res} \DeclareMathOperator{\Rd}{Rd}

\newcommand{\bK}{\mathbf{K}}
\newcommand{\bL}{\mathbf{L}}

\newcommand{\bS}{\mathbf{S}}

\newcommand{\cupdot}{\ensuremath{\mathaccent\cdot\cup}}

\begin{document}

\title{Finite simple additively idempotent semirings\tnoteref{t1}}
\tnotetext[t1]{This work has been supported by Science Foundation
  Ireland under grant no.\ 08/IN.1/I1950.}

\author{Andreas Kendziorra}
\ead{Andreas.Kendziorra@ucdconnect.ie}

\author{Jens Zumbr\"agel}
\ead{Jens.Zumbragel@ucd.ie} 

\address{Claude Shannon Institute\\
  University College Dublin\\
  Belfield, Dublin 4, Ireland}

\date{}

\begin{abstract}
  Since for the classification of finite (congruence-)simple semirings
  it remains to classify the additively idempotent semirings, we
  progress on the characterization of finite simple additively
  idempotent semirings as semirings of join-morphisms of a
  semilattice.  We succeed in doing this for many cases, amongst
  others for every semiring of this kind with an additively neutral
  element. As a consequence we complete the classification of finite
  simple semirings with an additively neutral element.  To complete
  the classification of all finite simple semirings it remains to
  classify some very specific semirings, which will be discussed here.
  Our results employ the theory of idempotent irreducible semimodules,
  which we develop further.
\end{abstract} 


\maketitle


\section{Introduction}

There have been several studies on simple semirings, e.g., in
\cite{bashir_2001, bashir, jezek, jezek2, mitchell, monico, zum}.
Amongst other things a complete classification of finite commutative
simple semirings was presented in \cite{bashir_2001}.  But there
exists so far no classification of all finite simple semirings.
Monico showed in \cite{monico} that every proper finite simple
semiring with more than two elements and nontrivial addition is
additively idempotent.  Thereupon additively idempotent semirings have
been studied in \cite{jezek, jezek2, zum}.

In this work we aim to describe all finite simple additively
idempotent semirings.  We did not succeed to characterize all these
semirings, but our approach covers many cases.  Before we go into more
detail we state the most important definitions.

\begin{defn}
  Let $R$ be a nonempty set and $+$ and $\cdot$ two binary operations
  on $R$.  Then $(R,+,\cdot)$ is called a \textit{semiring} if $(R,+)$
  is a commutative semigroup, $(R,\cdot)$ is a semigroup, and both
  distributive laws $r\cdot(s+t)=r\cdot s+r\cdot t$ and $(r+s)\cdot
  t=r\cdot t+ s\cdot t$ hold for all $r,s,t\in R$.
\end{defn}

If $(R,+,\cdot)$ is a semiring, then we write for the multiplication
mostly $xy:=x\cdot y$ for $x,y\in R$.

\begin{defn}
  Let $(R,+,\cdot)$ be a semiring and $r\in R$. We call $r$
  \textit{right} (\textit{left}) \textit{absorbing} if it is
  multiplicatively right (left) absorbing, i.e., if $sr=r$ ($rs=r$)
  holds for every $s\in R$. If $r$ is left and right absorbing then it
  is called \textit{absorbing}. If $(R,+)$ has a neutral element which
  is absorbing, then it is called the \textit{zero} of the semiring
  $(R,+,\cdot)$.
\end{defn}

\begin{defn}
  A \textit{congruence} of a semiring $(R,+,\cdot)$ is an equivalence
  relation $\sim $ on $R$ such that for every $r,s,t \in R$:
  \begin{displaymath}
    r\sim s\ \text{ implies }\ 
    r+t\sim s+t,\ tr\sim ts, \text{ and } rt\sim st \:.
  \end{displaymath} 
  The semiring $(R,+,\cdot)$ is called \textit{simple} if its only
  congruences are $\id_R$ and $R\times R$.
\end{defn}

To present the main result from \cite{monico}, we need the following
theorem about simple semigroups, which can be found in
\cite[Theorem~3.7.1]{howie}.

\begin{theorem}\label{simple_semigroups}
  Let $I=\{1,2,\dots,m\},\ J=\{1,2,\dots,n\}$, and $P=(p_{ij})$ an
  $n\times m$ matrix with $p_{i,j}\in\{0,1\}$ for all $i,j$ such that
  no row or column is identically zero, no two rows are identical, and
  no two columns are identical.  Let $S=(I\times J)\cup \{\infty\}$ and
  define a binary operation on $S$ by
  \begin{displaymath}
    (i,j)\cdot(k,l) := \begin{cases}
      (i,l) & \text{if }p_{jk} = 1 \:,\\
      \infty & \text{else} \:,
    \end{cases}
    \quad (i,j)\cdot \infty := \infty \cdot (i,j)
    := \infty\cdot\infty := \infty \:.
  \end{displaymath}
  Then $(S,\cdot)$ is a simple semigroup of order $mn+1$. Conversely,
  every finite simple semigroup with an absorbing element is
  isomorphic to one of this kind.
\end{theorem}

The main result in \cite{monico} is the following:

\begin{theorem}
  Let $(R,+,\cdot)$ be a finite simple semiring.  Then one of the
  following holds:
  \begin{enumerate}
  \item $|R|\le 2$,
  \item $(R,+,\cdot)\cong(\Mat_n(\mathbb{F}_q),+,\cdot)$ for some
    finite field $\mathbb{F}_q$ and some $n\ge 1$,
  \item $(R,+,\cdot)$ is a zero multiplication ring of prime order,
  \item $(R,\cdot)$ is a semigroup as in
    Theorem~\ref{simple_semigroups} with absorbing element $\infty\in
    R$ and $R+R=\{\infty\}$,
  \item $(R,+)$ is idempotent\label{additively_idempotent}.
\end{enumerate}
\end{theorem}

Of course every semiring in the first four cases is simple but not
every additively idempotent semiring is simple.  Hence, if one wants
to classify all finite simple semirings, then it remains to describe
all finite simple additively idempotent semirings.  The case, where
such a semiring has a zero, was studied in \cite{zum}.  We need some
preparation to state the main result of it.

A \textit{lattice} $\bL=(L,\leq)$ is an ordered set where for every
two elements $x,y\in L$ the supremum $x\vee y$ and the infimum
$x\wedge y$ in $L$ exists.  The lattice $\bL$ is called
\textit{complete} if for every subset $X\subseteq L$ the supremum
$\bigvee X$ and the infimum $\bigwedge X$ in $L$ exists.  A complete
lattice has a greatest element $1_\bL$ and a least element $0_\bL$.
If there is no confusion, then we write just $1$ and $0$.  A mapping
$f:L\rightarrow K$ between two lattices $\bL=(L,\leq)$ and
$\bK=(K,\leq)$ is called \textit{join-morphism} if $f(x\vee
y)=f(x)\vee f(y)$ holds for every $x,y\in L$.  By $\JM(\bL)$ we denote
the set of all join-morphisms from $\bL$ to $\bL$.  If $\bL$ and $\bK$
are complete and $f$ fulfills $f(\bigvee X)=\bigvee f(X)$ for every
subset $X\subseteq L$, then $f$ is called \textit{residuated} (or
\textit{complete join-morphism}).  If $\bL$ and $\bK$ are finite, then
$f$ is residuated iff it is a join-morphism and fulfills
$f(0_\bL)=0_\bK$.  By $\Res(\bL)$ we denote the set of all residuated
mappings from $\bL$ to $\bL$. The structure $(\Res(\bL),\vee,\circ)$,
where $\vee$ denotes the pointwise supremum and $\circ$ the
composition of two mappings, is a semiring for a complete lattice
$\bL$. For $a,b\in L$ define the mapping $e_{a,b}\in\Res(\bL)$ by
\begin{displaymath}
  e_{a,b}:L\rightarrow L \:,\quad x\mapsto \begin{cases}
    0\quad & \text{if }x\leq a \:,\\
    b&\text{else} \:.
  \end{cases}
\end{displaymath}
More information about lattices can be found in \cite{birkhoff} and
about residuated mappings in \cite{janowitz}.  The main result from
\cite{zum} can be stated as follows:

\begin{theorem}\label{theorem_zum}
  Let $\bL=(L,\leq)$ be a finite lattice and $(R,\vee, \circ)$ a
  subsemiring of $(\Res(\bL),\vee, \circ)$ such that $e_{a,b}\in R$
  for every $a,b\in L$. Then $(R,\vee,\circ)$ is a finite simple
  additively idempotent semiring with zero.  Conversely, every finite
  simple additively idempotent semiring $(S,+,\cdot)$ with $|S|>2$ and
  a zero is isomorphic to such a semiring.
\end{theorem} 

We use here the same approach as in \cite{zum}, i.e., we try to
characterize every finite simple additively idempotent semiring as a
semiring of join-morphisms of a semilattice.  For this we have to
distinguish between several cases.  By the \emph{greatest element} of
a finite additively idempotent semiring $(R,+,\cdot)$ or a finite
commutative semigroup $(R,+)$ we mean the greatest element
of~$(R,\le)$, where the order $\le$ on $R$ is defined by $x\le
y\,:\Leftrightarrow\,x+y=y$ for $x,y\in R$; this element is
$\sum_{r\in R}r$.  We consider the cases where the greatest element of
a semiring is
\begin{enumerate}
\item neither right nor left absorbing,
\item right but not left absorbing,
\item left but not right absorbing,
\item absorbing and the semiring possesses a finite idempotent
  irreducible semimodule (see Section~\ref{chapter_semimodules} for
  definitions) which
  \begin{enumerate}
  \item satisfies property $(*)$ below,
  \item does not satisfy property $(*)$.
    \label{case_remaining} 
  \end{enumerate}
\end{enumerate}

For a finite idempotent semimodule $(M,+)$ with greatest element $\infty_M$
we consider the property:
\begin{displaymath}\tag{$*$}
  \exists\,u\in M\ \
  \forall\,x\in M\!\setminus\!\{\infty_M\}:\
  \infty_M\ne u+x \:.
\end{displaymath}
This property is satisfied, e.g., if $(M,+)$ has a neutral element or
if $\infty_M$ is \emph{join-irreducible}, i.e., if $\infty_M\ne u+x$
for all $u,x\in M\!\setminus\!\{\infty_M\}$.

We succeed with this approach for every case, except
Case~\ref{case_remaining}, for which we have a conjecture.  As
semirings in Case~\ref{case_remaining} have no additively neutral
element we complete with our characterization theorems the
classification of finite simple semirings with additively neutral
element, which will be summarized in
Theorem~\ref{classification_add_neutr}.

This paper is structured as follows.
Section~\ref{chapter_semimodules} contains a comprehensive study of
semimodules, especially idempotent irreducible semimodules.  These
semimodules are necessary to describe the embedding of a finite simple
additively idempotent semiring into the semiring of join-morphisms of
a semilattice, which is done in Section~\ref{section_embedding}.  In
Section~\ref{chapter_subsemirings} we study simple subsemirings of a
semiring of join-morphisms of a semilattice.  The main results are
stated in Section~\ref{chapter_main_results}, which are the
characterization theorems that characterize a finite simple additively
idempotent semiring as a semiring of join-morphisms of a semilattice.
Section~\ref{chapter_isomorphic_semirings} clarifies that if two
semirings considered in the main results are isomorphic, then also the
corresponding semilattices have to be isomorphic.  The question when a
semiring has an additively or multiplicatively neutral element is
answered in Section~\ref{chapter_neutral_elements}, where we also
state the complete classification of finite simple semirings with an
additively neutral element.  In Section~\ref{chapter_remaining_case}
we discuss the remaining Case~\ref{case_remaining}, and at last we
present some examples in Section~\ref{chapter_example}.


\section{Semimodules}\label{chapter_semimodules}

In the following let $(R,+,\cdot)$ be a semiring.

\begin{defn}
  An $R$\emph{-semimodule} is a commutative semigroup $(M,+)$
  together with an $R$-multiplication $R\times M\rightarrow M$,
  $(r,x)\mapsto rx$, such that for all $r,s \in R$ and $x,y\in M$ it
  holds:
  \begin{align*}
    r(sx)=(rs)x\:,\quad (r+s)x=rx+sx \:,\ \text{\ and }\ r(x+y)=rx+ry \:.
  \end{align*}
\end{defn}

For an $R$-semimodule $(M,+)$ and a subset $N\subseteq M$ we define
$RN := \{rn\mid r\in R,\ n\in N\}$, and for $a\in M$ we define $Ra :=
\{ra\mid r\in R\}$.

\begin{defn}
  Let $(M,+)$ be an $R$-semimodule.  An \emph{($R$-)subsemimodule} of
  $(M,+)$ is a subsemigroup $(N,+)$ of $(M,+)$ such that $RN\subseteq
  N$.
\end{defn}

If $(M,+)$ is an $R$-semimodule and $a\in M$, then $(Ra,+)$ is clearly
a subsemimodule of $(M,+)$.

\begin{defn}
  Let $(M,+)$ be an $R$-semimodule.  A \emph{(semimodule)
    congruence} on $(M,+)$ is an equivalence relation $\sim $ on $M$
  such that
  \begin{displaymath}
    x\sim y\quad \text{implies}\quad 
    x+z\sim y+z\ \text{ and }\ rx\sim ry\:,
  \end{displaymath}
  for all $x,y,z\in M$ and all $r\in R$.
\end{defn}

For a congruence $\sim$ on an $R$-semimodule $(M,+)$ note that
$(M/\!\!\sim\,,+)$ with $[x]+[y]:=[x+y]$ and $r[x]:=[rx]$ is again an
$R$-semimodule, which is called \emph{quotient semimodule}.

If $(M,+)$ is an $R$-semimodule then we denote by $\End(M,+)$ the set
of all semigroup endomorphisms of the semigroup $(M,+)$.

\begin{defn}\label{def_T}
  An $R$-semimodule $(M,+)$ is called \emph{faithful} if the
  semiring homomorphism
  \begin{displaymath}
    T: R\rightarrow \End(M,+) \:, \quad r\mapsto T_r\quad \text{with}
    \quad T_r:x\mapsto rx \:,
  \end{displaymath}
  is injective.  Furthermore, $(M,+)$ is said to be \emph{faithful
    of smallest cardinality} if $(M,+)$ is faithful and any
  $R$-semimodule $(N,+)$ with cardinality $|N|<|M|$ is not faithful.
\end{defn}

\begin{defn}
  An $R$-semimodule $(M,+)$ is called \emph{trivial} if $|M|=1$.
  Furthermore, following~\cite{qt-1}, we call an $R$-semimodule
  $(M,+)$ \emph{quasitrivial} if the homomorphism $T:R\to\End(M,+)$
  from Definition~\ref{def_T} is constant, i.e., if $rx=sx$ for all
  $r,s\in R$, $x\in M$, and we call $(M,+)$ \emph{id-quasitrival} if
  $T_r = \id_M$ for all $r\in R$, i.e., if $rx=x$ for all $r\in R$,
  $x\in M$.
\end{defn}

Clearly, a trivial semimodule is id-quasitrivial, and an
id-quasitrivial semimodule is quasitrivial.

\begin{remark}\label{remark_embed}
  Let $(R,+,\cdot)$ be simple, and let $(M,+)$ be a non-quasitrivial
  $R$-semimodule.  Then $(M,+)$ is faithful and $(R,+,\cdot)$ is
  isomorphic to the subsemiring $(T(R),+,\circ)$ of
  $(\End(M,+),+,\circ)$.
\end{remark}

\begin{defn}
  An $R$-semimodule $(M,+)$ is called \emph{sub-irreducible}, if it is
  non-quasitrivial and it has only id-quasitrivial proper
  subsemimodules; it is called \emph{quotient-irreducible} if it is
  non-quasitrivial and possesses only the trivial proper quotient
  semimodule, i.e., if its only congruences are $\id_M$ and
  $M\times M$.  If an $R$-semimodule is both sub- and
  quotient-irreducible it is called \emph{irreducible}.
\end{defn}


\begin{conjecture}
  Let $(R,+,\cdot)$ be finite, simple, and additively idempotent and
  $(M,+)$ a finite idempotent $R$-semimodule. Then
  $(M,+)$ is sub-irreducible iff it is quotient-irreducible.
\end{conjecture}

Evidence for this conjecture is given by the fact that all semimodules
considered in experiments satisfy this equivalence.


\subsection{Existence of idempotent irreducible semimodules}%
\label{chapter_existence}

Let $(R,+,\cdot)$ be in this section a finite simple semiring. The
main result of this section is
Proposition~\ref{existence_of_irreducible_semimod}, which states that
$(R,+,\cdot)$ admits a finite idempotent irreducible semimodule if
$(R,+,\cdot)$ is additively idempotent and fulfills $|R|>2$.

If a commutative semigroup $(S,+)$ (e.g., a semimodule or the additive
semigroup of a semiring) is idempotent, then it is a
semilattice. Hence, it can be regarded as an ordered set $(S,\leq)$,
where $x\leq y :\Leftrightarrow x+y=y$ for all $x,y\in S$. If $(S,+)$
has a neutral element $0_S$, then $0_S$ is the least element in
$(S,\leq)$.  Conversely, a least element in $(S,\leq)$ is a neutral
element in $(S,+)$. If $S$ is finite then $(S,\leq)$ has a least
element iff $(S,\leq)$ is a lattice.

A finite semilattice has a greatest element. To avoid confusion with
multiplicatively neutral elements, we denote the greatest element of a
semilattice $(S,+)$, which is an idempotent semimodule or the
idempotent additive structure of a semiring, by $\infty_S$ or just by
$\infty$ if it is clear to which semilattice $\infty$ belongs.

If for elements $a,b\in P$ with $a< b$ in an ordered set
$(P,\leq)$ it holds that
\begin{displaymath}
  a\leq c\leq b\quad\text{implies}\quad a=c\ \text{ or }\ c=b \:,
\end{displaymath}
then $a$ is called a \emph{lower neighbor} of $b$ and $b$ an
\emph{upper neighbor} of $a$.  An element $m\in P$ is called
\emph{minimal} in $(P,\leq)$ if there is no element $n\in P$ with
$n<m$.  Let $\Min(P,\leq)$ denote the set of minimal elements in
$(P,\leq)$.  Furthermore, for $x\in S$ denote $x_\downarrow:=\{y\in
S\mid y\leq x\}$ and $x^\uparrow:=\{y\in S\mid y\geq x\}$.

The following lemma is \cite[Lemma 6]{monico}.

\begin{lemma}\label{monico}
  If the multiplication table of $(R,+,\cdot)$ has two identical rows
  or two identical columns, then $|RR|=1$ or $|R|=2$.
\end{lemma}

\begin{lemma}\label{lemma_existence_of_faithful_semimod}
  Let $|RR|>1$ and $|R|>2$.  Then $(R,+)$ as an $R$-semimodule is not
  quasitrivial.  Therefore, there exists a faithful $R$-semimodule
  $(M,+)$ of smallest cardinality and $|M|\le |R|$.
\end{lemma}

\begin{proof}
  Suppose $(R,+)$ is quasitrivial, then $rt=st$ for all $r,s,t\in
  R$.  Thus, every two rows in the multiplication table of
  $(R,+,\cdot)$ are identical, in contradiction to Lemma~\ref{monico}.
  Therefore, $(R,+)$ is faithful and the rest is clear.
\end{proof}

If $(M,+)$ is a faithful $R$-semimodule of smallest cardinality, then
all proper sub- and quotient-semimodules of $(M,+)$ are non-faithful,
and therefore quasitrivial.  We will show in
Proposition~\ref{prop_M_is_irreducible} that $(M,+)$ is even
irreducible.

\begin{lemma}\label{lemma_ra=sa}
  Let $|RR|>1$, $|R|>2$, let $(M,+)$ be an $R$-semimodule and let
  $a\in M$ such that the subsemimodule $(Ra,+)$ is quasitrivial.  Then
  $|Ra|=1$.
\end{lemma}

\begin{proof}
  On the $R$-semimodule $(R,+)$ consider the congruence $\sim_a$
  defined by
  \begin{displaymath}
    r\sim_a s \quad :\Leftrightarrow \quad ra=sa
  \end{displaymath}
  for $r,s\in R$.  Since $(Ra,+)$ is quasitrivial, for all $r,s,t\in
  R$ it holds that $(rt)a=r(ta)=s(ta)=(st)a$, so that $rt\sim st$, and
  $\sim_a$ is even a semiring congruence on $(R,+,\cdot)$.  Supposing
  $\sim_a \:= \id_R$, we have $rt=st$ for all $r,s,t\in R$,
  contradicting Lemma~\ref{lemma_existence_of_faithful_semimod}.  By
  simplicity of $(R,+,\cdot)$, then $\sim_a \:= R\times R$, so that
  $|Ra|=1$.
\end{proof}

\begin{lemma}\label{lemma_Ra=M}
  Let $|RR|>1$, $|R|>2$, and let $(M,+)$ be a faithful $R$-semimodule
  of smallest cardinality.  Then there exists $a\in M$ with $Ra=M$.
\end{lemma}

\begin{proof}
  Assume on the contrary that for any $a\in M$ it holds that
  $Ra\subsetneq M$, then the $R$-semimodule $(Ra,+)$ is not faithful
  and, by simplicity of $(R,+,\cdot)$, it is quasitrivial, hence
  Lemma~\ref{lemma_ra=sa} implies that $|Ra|=1$.  This means that
  $(M,+)$ is quasitrivial, which contradicts the assumption that
  $(M,+)$ is faithful.
\end{proof}

\begin{lemma}\label{lemma_quotient-irr}
  Let $|RR|>1$, $|R|>2$, and let $(M,+)$ be a faithful $R$-semimodule
  of smallest cardinality.  Then $(M,+)$ is quotient-irreducible.
\end{lemma}

\begin{proof}
  Let $\sim$ be a semimodule congruence on $(M,+)$ distinct from
  $\id_M$ and let $N:=M/\!\!\sim$.  Then $(N,+)$ is not faithful and
  therefore it holds that $[rx]=r[x]=s[x]=[sx]$ for all $r,s\in R$
  and $x\in M$. By Lemma~\ref{lemma_Ra=M} there exists $a\in M$ with
  $Ra=M$. Since $[ra]=[sa]$ for all $r,s\in R$, we see that $\sim \:=
  M\times M$ and the statement follows.
\end{proof}

\begin{proposition}\label{prop_M_is_irreducible}
  Let $|RR|>1$, $|R|>2$, and let $(M,+)$ be a faithful $R$-semi\-module
  of smallest cardinality. Then $(M,+)$ is irreducible.
\end{proposition}

\begin{proof}
  From Lemma~\ref{lemma_quotient-irr} we know that $(M,+)$ is
  quotient-irreducible.  Now let $(N,+)$ be a proper subsemimodule of
  $(M,+)$, which has to be quasitrivial, so that $|Rn|=1$ for every
  $n\in N$.  Define the equivalence relation $\sim $ on $M$ by
  \begin{displaymath}
    x\sim y \quad :\Leftrightarrow \quad \forall r\in R:\ rx=ry
  \end{displaymath}
  for all $x,y\in M$. Let $x,y,z\in M$ with $x\sim y$ and let
  $r,s\in R$.  It holds that $r(x+z)=rx+rz=ry+rz=r(y+z)$, i.e.,
  $x+z\sim y+z$.  It also holds that $r(sx)=(rs)x=(rs)y=r(sy)$, i.e.,
  $sx\sim sy$. Thus, $\sim$ is a semimodule congruence on $(M,+)$.
  Lemma~\ref{lemma_quotient-irr} implies that $\sim \:= M\times M$ or
  $\sim \:= \id_M$.  Assume that $\sim \:= M\times M$.  Then for all
  $x,y\in M$ and all $r\in R$ it holds that $rx=ry$, so in particular
  $Rx=Ry$.  By Lemma~\ref{lemma_Ra=M} there exists $a\in M$ with
  $Ra=M$.  Then $M=Ra=Rn$ for all $n\in N$, and in particular
  $|M|=|Rn|=1$, a contradiction.

  It must hold that $\sim \:= \id_M$.  For every $n\in N$ there is
  $f(n)\in N$ such that $Rn=\{f(n)\}$.  Hence if $f(n_1)=f(n_2)$ for
  some $n_1, n_2\in N$ then $rn_1=rn_2$ for all $r\in R$, so that
  $n_1\sim n_2$ and thus $n_1=n_2$.  Now for any $n\in N$ and $r,s\in
  R$ it holds that $f(f(n)) = r(f(n)) = r(sn) = (rs)n = f(n)$, so that
  $f(n)=n$ follows.  Thus for every $n\in N$ it holds that $Rn=\{n\}$,
  which means that $(N,+)$ is id-quasitrivial.  So we have proven that
  $(M,+)$ is sub-irreducible.
\end{proof}

\begin{lemma}\label{|RR|>1}
  Let $(R,+,\cdot)$ be additively idempotent and $|R|>2$. Then it
  holds that $|RR|>1$.
\end{lemma}

\begin{proof}
  Let $x\in R\setminus\{\infty\}$ and let $\sim$ be the equivalence
  relation on $R$ with the equivalence classes $x_\downarrow$ and
  $R\setminus x_\downarrow$. It is easy to check that $\sim $ is a
  nontrivial congruence of the semigroup $(R,+)$. If $|RR|=1$ would
  hold, then every equivalence relation on $R$ would be a congruence
  of $(R,\cdot)$. Consequently, $\sim$ would be a nontrivial
  congruence of $(R,+,\cdot)$, contradicting the simplicity of
  $(R,+,\cdot)$.
\end{proof}

By this last result, when considering a semiring $(R,+,\cdot)$ that
fulfills $|R|>2$ and ${|RR|>1}$, we can drop the condition $|RR|>1$ in
the case of an additively idempotent semiring.

\begin{lemma}\label{lemma_M_is_idempotent}
  Let $(R,+,\cdot)$ be additively idempotent, $|R|>2$, and let $(M,+)$
  be a faithful $R$-semimodule of smallest cardinality. Then $(M,+)$
  is idempotent.
\end{lemma}

\begin{proof}
  Lemma~\ref{lemma_Ra=M} yields the existence of an element $a\in M$
  with $Ra=M$. Let $b\in M$. Then there exists $r\in R$ with $ra=b$
  and it follows that $b+b=ra+ra=(r+r)a=ra=b$. Thus, $(M,+)$ is
  idempotent.
\end{proof}

\begin{proposition}\label{existence_of_irreducible_semimod}
  Let $(R,+,\cdot)$ be additively idempotent with $|R|>2$.  Then there
  exists a finite idempotent irreducible $R$-semimodule.
\end{proposition}

\begin{proof}
  By Lemma~\ref{lemma_existence_of_faithful_semimod} there exists a
  faithful $R$-semimodule of smallest cardinality, which is
  irreducible by Proposition~\ref{prop_M_is_irreducible} and
  idempotent by Lemma~\ref{lemma_M_is_idempotent}.
\end{proof}


\subsection{Properties of idempotent sub-irreducible 
  semimodules}%
\label{section_properties}

Let throughout this section $(R,+,\cdot)$ be a finite simple
additively idempotent semiring with $|R|>2$ and $(M,+)$ a finite
idempotent sub-irreducible $R$-semi\-module.  We will study the
properties of the $R$-semimodule $(M,+)$, depending on the properties
of $(R,+,\cdot)$ ($\infty_R$ is absorbing, $0_R$ exists and is left
absorbing etc.). These properties are needed to describe the embedding
of $(R,+,\cdot)$ into $(\JM(\bL),\vee,\circ)$ for a suitable
semilattice $\bL$, what will be done in
Section~\ref{section_embedding}.

We start with a preliminary lemma.

\begin{lemma}
  Let $(S,+,\cdot)$ be an additively idempotent semiring and $(N,+)$
  an idempotent $S$-semimodule.  For all $r,s\in S$ and $x,y\in N$ it
  holds:
  \begin{displaymath}
    x\le y \ \Rightarrow\ rx\le ry \qquad \text{and} \qquad
    r\le s \ \Rightarrow\ rx\le sx \:.
  \end{displaymath}
\end{lemma}

\begin{proof}
  If $x\le y$, i.e., $x+y=y$, then $rx+ry=r(x+y)=ry$, i.e., $rx\le
  ry$.  Similarly, if $r\le s$, i.e., $r+s=s$, then $rx+sx=(r+s)x=sx$,
  i.e. $rx\le sx$.
\end{proof}

\begin{lemma}\label{lemma_a=b}
  Let $M$ be an idempotent sub-irreducible $R$-semimodule, and let
  $a,b\in M$ such that $Ra=\{b\}$.  Then $a=b$, and $a$ is either an
  absorbing or a neutral element of $(M,+)$.
\end{lemma}

\begin{proof}
  Consider the set $N := \{x\in M\mid Rx=\{b\}\}$, which contains the
  element~$a$.  Then $(N,+)$ is a subsemimodule of $(M,+)$.  Indeed,
  let $x,y\in N$ and let $s\in R$.  For all $r\in R$ we have $r(x+y) =
  rx+ry = b+b = b$ and $r(sx) = (rs)x = b$, hence $x+y\in N$ and
  $sx\in N$, as desired.  Since $(M,+)$ is non-quasitrivial we have
  $|RM|>1$, so that $N\ne M$, and hence $(N,+)$ is id-quasitrivial.
  In particular, for $a\in N$ we have $Ra=\{a\}$, so that $a=b$.

  Now consider the sets $a_\downarrow$ and $a^\uparrow$, which form
  $R$-subsemimodules $(a_\downarrow,+)$ and $(a^\uparrow,+)$ of
  $(M,+)$.  We have to show that either $a_\downarrow\, = M$, in which
  case $a$ is an absorbing element, or $a^\uparrow\, = M$, in which
  case $a$ is a neutral element.

  Suppose then that $a_\downarrow\,\ne M$ and $a^\uparrow\,\ne M$.
  Since $(M,+)$ is sub-irreducible we have that $(a_\downarrow,+)$ and
  $(a^\uparrow,+)$ are id-quasitrivial.  We claim that $(M\setminus
  a_\downarrow,+)$ is an $R$-subsemimodule of $(M,+)$ as well.  Let
  $x,y\in M$, $x,y\notin a_\downarrow$ and let $r\in R$; then clearly
  $x+y\notin a_\downarrow$.  Suppose that $rx\in a_\downarrow$, i.e.,
  $rx\le a$.  Then, since $x+a\in a^\uparrow$, it holds that $x+a =
  r(x+a) = rx+ra = rx+a = a$, so that $x\le a$, contradicting $x\notin
  a_\downarrow$.  Hence $(M\setminus a_\downarrow,+)$ is an
  $R$-subsemimodule of $(M,+$), which is proper and thus
  id-quasitrivial.  From this and because $M = a_\downarrow \cup\,
  (M\setminus a_\downarrow)$ it follows that $(M,+)$ is
  id-quasitrivial, which contradicts a requirement for
  sub-irreducibility.
\end{proof}

\begin{corollary}\label{corollary_id-quasitrivial}
  Let $(N,+)$ be a proper, hence id-quasitrivial, subsemimodule of
  $(M,+)$.  If $(M,+)$ has no neutral element then $N = \{\infty\}$.
  If $(M,+)$ has a neutral element $0$ then it holds that $N\subseteq
  \{0,\infty\}$.
\end{corollary}

\begin{proposition}\label{prop_Ra=M}
  There exists $a\in M$ with $Ra=M$.  Furthermore, if $x\in M$
  satisfies $Rx\ne M$ then either $x=\infty_M$ and $R\infty_M =
  \{\infty_M\}$, or a neutral element $0_M\in M$ exists, $x=0_M$, and
  $R0_M = \{0_M\}$.
\end{proposition}

\begin{proof}
  Let $x\in M$ be such that $Rx\ne M$.  Then the subsemimodule
  $(Rx,+)$ of $(M,+)$ is id-quasitrivial, and by
  Lemma~\ref{lemma_ra=sa} it follows that $|Rx|=1$.  Now
  Lemma~\ref{lemma_a=b} implies that either $x=\infty$ and
  $R\infty_M=\{\infty_M\}$, or $x=0_M$ and $R0_M=\{0_M\}$, which
  proves the last statement.

  For the first statement, suppose that for every $a\in M$ it holds
  that ${Ra\ne M}$, so that $|Ra|=1$.  Then $(M,+)$ is
  quasitrivial, which is a contradiction.
\end{proof}

\begin{proposition}\label{prop_greatest_element}
  The following holds:
  \begin{enumerate}
  \item if $\infty_Rx\ne \infty_M$ for some $x\in M$, then $(M,+)$ has
    a neutral element $0_M$, $x=0_M$, and $R0_M=\{0_M\}$,%
    \label{infty_1}
  \item the element $\infty_R$ is right absorbing iff
    $R\infty_M=\{\infty_M\}$,\label{infty_2}
  \item if $\infty_R$ is not left absorbing then $0_M\in M$ and
    $R0_M=\{0_M\}$,\label{infty_3}
  \item if $\infty_R$ is left absorbing, $(M,+)$ is
    quotient-irreducible, $|M|>2$, and ${0_M\in M}$ then $R0_M =
    M$.\label{infty_4}
  \end{enumerate}
\end{proposition}

\begin{proof}
  \ref{infty_1}.: If $x\in M$ satisfies $Rx=M$ then there exists $r\in
  R$ with $rx=\infty_M$ and hence $\infty_Rx \ge rx = \infty_M$.
  Thus the statement follows from Proposition~\ref{prop_Ra=M}.

  \ref{infty_2}.: If $\infty_R$ is right absorbing then $r\infty_M =
  r(\infty_R\infty_M) = (r\infty_R)\infty_M = \infty_R\infty_M =
  \infty_M$ for all $r\in R$.  Conversely, if $R\infty_M =
  \{\infty_M\}$, let $x\in M$.  If $\infty_Rx = \infty_M$ holds then
  $(r\infty_R)x = r(\infty_Rx) = r\infty_M = \infty_M = \infty_Rx$.
  Otherwise, by~\ref{infty_1}.\ we must have $0_M\in M$, $x=0_M$,
  and $R0_M=\{0_M\}$, and thus $(r\infty_R)0_M = r(\infty_R 0_M) =
  r0_M = 0_M = \infty_R 0_M$.  Hence, $(r\infty_R)x = \infty_Rx$ for
  all $x\in M$, so that, since $(M,+)$ is faithful, $\infty_R$ is
  right absorbing.

  \ref{infty_3}.: If $\infty_Rx = \infty_M$ for all $x\in M$ then
  $(\infty_Rr)x = \infty_R(rx) = \infty_R = \infty_Rx$ for all $x\in
  M$, and, since $(M,+)$ is faithful, $\infty_R$ would be
  left absorbing, contradicting the precondition.  Therefore,
  by~\ref{infty_1}.\ we have $0_M\in M$ and $R0_M = \{0_M\}$.

  \ref{infty_4}.: Assume on the contrary that $R0_M\ne M$, so that
  $R0_M = \{0_M\}$ by Proposition~\ref{prop_Ra=M}, and hence
  $\infty_R0_M = 0_M$.  By~\ref{infty_1}., for all $x\in
  M\setminus\{0\}$ we have $\infty_R = \infty_Rx = (\infty_Rr)x =
  \infty_R(rx)$, so that $rx\ne 0$.  It follows easily that
  the equivalence relation on $M$ with classes $\{0_M\}$ and
  $M\setminus\{0_M\}$ is a nontrivial semimodule congruence,
  contradicting the quotient-irreducibility.
\end{proof}

\begin{proposition}\label{prop_semirings_with_neutral_element}
  Let $(R,+)$ have a neutral element $0_R$. Then $(M,+)$ has a neutral
  element.
\end{proposition}

\begin{proof}
  By Proposition~\ref{prop_Ra=M} there exists $a\in M$ with $Ra=M$.
  For all $r\in R$ it holds that $0_R\le r$ and thus $0_Ra\le ra$,
  therefore $0_Ra$ is the least element, i.e., the neutral element, in
  $(M,+)$.
\end{proof}

\begin{proposition}\label{prop_has_a_zero}
  Let $\infty_R$ be neither left nor right absorbing.  Then
  $(R,+,\cdot)$ has a zero.
\end{proposition}

\begin{proof}
  Since $\infty_R$ is not left absorbing it follows from
  Proposition~\ref{prop_greatest_element}-\ref{infty_3}.\ that $(M,+)$
  has a neutral element $0_M$ such that $R0_M=\{0_M\}$.  Since
  $\infty_R$ is not right absorbing,
  Proposition~\ref{prop_greatest_element}-\ref{infty_2}.\ implies
  $R\infty_M\ne\{\infty_M\}$ and Proposition~\ref{prop_Ra=M} implies
  $R\infty_M=M$.  Hence, there exists $r\in R$ such that $r\infty_M =
  0_M$.  Then we have $rx\le r\infty_M = 0_M$ for all $x\in M$.  Since
  $(M,+)$ is faithful, $r = 0_R$ must be the neutral element of
  $(R,+)$, and $0_Rx = 0_M$ for all $x\in M$.

  For all $r\in R$ and $x\in M$ it follows $(0_Rr)x = 0_R(rx) = 0_M =
  0_Rx$ and $(r0_R)x = r(0_Rx) = r0_M = 0_M = 0_Rx$, so that, since
  $(M,+)$ is faithful, $0_Rr = 0_R = r0_R$.  Hence, $0_R$ is left and
  right absorbing and therefore a zero.
\end{proof}



\subsection{Density results for idempotent irreducible 
  semimodules}

Let first $(R,+,\cdot)$ be any semiring.

\begin{lemma}\label{lemma_uv}
  Let $(M,+)$ be an idempotent $R$-semimodule and let $u,v\in M$ be
  such that $u$ is minimal, $u<v$, and $v\ne u+x$ for all $x\in
  M\setminus\{v\}$.  Define the (reflexive, symmetric) relation $\rho$
  on $M$ by
  \[ a \,\rho\, b \quad :\Leftrightarrow \quad \forall r\in R:\ \{ra,rb\}
  \ne \{u,v\} \,, \] and let $\sim$ be the transitive hull of $\rho$.
  Then the equivalence relation $\sim$ is a congruence on $M$.
\end{lemma}

\begin{proof}
  Let $a,b,c\in M$ with $a\,\rho\,b$, then we claim that
  $(c+a)\,\rho\,(c+b)$.  Otherwise, there exists $r\in R$ with
  $\{r(c+a),r(c+b)\} = \{u,v\}$.  If, say, $rc+ra = r(c+a) = u$ and
  $rc+rb = r(c+b) = v$, it follows from the minimality of $u$ that
  $rc=ra=u$ and from the condition on $v$ that $rb=v$, hence
  $\{ra,rb\} = \{u,v\}$, contradicting $a\,\rho\,b$.  Similarly, if
  $r(c+b) = u$ and $r(c+a) = v$ we infer that $rb = u$ and $ra = v$,
  again a contradiction.  From this it follows that $a\sim b$ implies
  $c+a\sim c+b$, for all $a,b,c\in M$.

  Now let $a,b\in M$ with $a\,\rho\,b$ and let $s\in R$.  Then $sa
  \,\rho\, sb$, since for all $r\in R$ it holds that $\{rsa,rsb\} \ne
  \{u,v\}$.  Therefore, for all $a,b\in M$ and $s\in R$ we have that
  $a\sim b$ implies $sa\sim sb$.
\end{proof}

\begin{proposition}\label{prop_uv}
  Let $(M,+)$ be a quotient-irreducible idempotent $R$-semi\-module
  and let $u,v\in M$ be such that $u$ is minimal, $u<v$, and $v\ne
  u+x$ for all $x\in M\setminus\{v\}$.  Suppose that either
  \begin{enumerate}
  \item $Ru = \{u\}$ and $v\in Rx$ for all $x\in M\setminus\{u\}$, or
  \item $Rv = \{v\}$ and $u\in Rx$ for all $x\in M\setminus\{v\}$.
  \end{enumerate}
  Then for all $a,b\in M$ with $b\not\le a$ there exists $r\in R$ such
  that $ra=u$ and $rb=v$.
\end{proposition}

\begin{proof}
  Consider the congruence $\sim$ on $M$ of Lemma~\ref{lemma_uv}.
  Suppose that Condition~1.\ holds and assume that $u\sim z$ for some
  $z\in M\setminus\{u\}$.  Then also $u \,\rho\, x$ for some $x\in
  M\setminus\{u\}$.  By the condition there exists $r\in R$ such that
  $ru=u$ and $rx=v$, hence $\{ru,rx\} = \{u,v\}$, contradicting $u
  \,\rho\, x$.  Therefore, $\sim \:= M\times M$ cannot hold, and by
  quotient-irreducibility of $M$ it follows that $\sim \:= \id_M$.

  Similarly, $v\sim z$ does not hold under Condition~2.\ for any $z\in
  M\setminus\{v\}$, and it follows that $\sim \:= \id_M$ in this case
  as well.

  Now let $a,b\in M$ with $b\not\le a$.  Then $a<a+b$ and since $\sim
  \:= \id_M$ there exists $r\in R$ such that $\{ra,r(a+b)\} =
  \{u,v\}$.  Since $ra\le r(a+b)$ we have $ra=u$ and $ra+rb = r(a+b) =
  v$, and from the condition on $v$ it follows that $rb=v$.
\end{proof}

Now let $(R,+,\cdot)$ be a finite simple additively idempotent
semiring with $|R|>2$, and let $(M,+)$ be a finite idempotent
irreducible $R$-semimodule.  The following two propositions are
density results akin to~\cite[Proposition~3.13]{zum}.

Let $a,b\in M$.  If there exists an
element $r\in R$, with
\begin{displaymath}
  r x=\begin{cases}
    b & \text{if }x\leq a \:,\\
    \infty_M\quad & \text{otherwise} \:,
  \end{cases}
\end{displaymath}
then it is unique, since $(M,+)$ is faithful, and we denote it by
$r_{a,b}$.  

\begin{proposition}\label{prop_r_a0}
  Let $\infty_R$ be not left absorbing.  Then $r_{a,0_M}\in R$ for
  every $a \in M\setminus\{\infty_M\}$.
\end{proposition}

\begin{proof}
  By Proposition~\ref{prop_greatest_element}, $(M,+)$ has a neutral
  element $0_M$ and it holds that $R0_M=\{0_M\}$.  Also, $Rx = M$ for
  all $x\in M\setminus\{0_M,\infty_M\}$ by Proposition~\ref{prop_Ra=M}
  and $\infty_M\in R\infty_M$, hence $\infty_M\in Rx$ for all $x\in
  M\setminus\{0_M\}$.  Therefore, we can apply
  Proposition~\ref{prop_uv} using Condition~1.\ with $u=0_M$ and
  $v=\infty_M$.  

  For fixed $a\in M\setminus\{\infty_M\}$ we conclude that for any
  $x\in M$ with $x\not\le a$ there exists $s_x\in R$ such that
  $s_xa=0_M$ and $s_xx=\infty_M$.  Define now $s:=\sum_{x\in
    M,\,x\not\le a}s_x\in R$ and let $z\in M$.  If $z\le a$, then
  $s_xz\le s_xa=0_M$ for every $x$ with $x\not\le a$, i.e., $sz=0_M$.
  If $z\not\le a$, then $sz=\sum_{x\in M,\,x\not\le a}s_xz\geq
  s_zz=\infty_M$, i.e., $sz=\infty_M$. Thus, $r_{a,0_M}=s\in R$.
\end{proof}

\begin{proposition}\label{prop_r_au}
  Suppose there exists $u\in \Min(M,\le)$ such that $\infty_M \ne u+x$
  for all $x\in M\setminus\{\infty_M\}$ and $R\infty_M=\{\infty_M\}$.
  Then $r_{a,u}\in R$ for every $a \in M\setminus\{\infty_M\}$.
\end{proposition}

\begin{proof}
  By Proposition~\ref{prop_Ra=M}, if $(M,+)$ has no neutral element
  then $Rx=M$ for all $x\in M\setminus\{\infty_M\}$, and if $(M,+)$
  has a neutral element $u=0_M$ then $Rx=M$ for all $x\in
  M\setminus\{0_M,\infty_M\}$ and $0_M\in R0_M$.  Hence $u\in Rx$ for
  all $x\in M\setminus\{\infty_M\}$ and we can apply
  Proposition~\ref{prop_uv} using Condition~2.\ with $u$ and
  $v=\infty_M$.  

  For fixed $a\in M\setminus\{\infty_M\}$ we conclude that for any
  $x\in M$ with $x\not\le a$ there exists $s_x\in R$ such that
  $s_xa=u$ and $s_xx=\infty_M$.  As in the proof of
  Proposition~\ref{prop_r_a0}, we define $s:=\sum_{x\in M,\,x\not\le
    a}s_x\in R$ and we have $sz=u$ for $z\le a$ and $sz=\infty_M$
  otherwise.  Thus, $r_{a,u}=s\in R$.
\end{proof}


\section{Embedding of $(R,+,\cdot)$ into $(\JM(\bL),\vee,\circ)$}%
\label{section_embedding}

In this section let $(R,+,\cdot)$ be again a finite simple additively
idempotent semiring with $|R|>2$.

We are going to embed $(R,+,\cdot)$ into the semiring
$(\JM(\bL),\vee,\circ)$ for a suitable finite semilattice
$\bL=(L,\leq)$.  The subsemiring $(S,\vee,\circ)$ of
$(\JM(\bL),\vee,\circ)$ corresponding to $(R,+,\cdot)$ fulfills then
certain conditions, depending on the properties of $(R,+,\cdot)$.  We
list in the beginning of this section all conditions that may arise
for $(S,\vee,\circ)$ and which may be necessary for the
characterization of $(R,+,\cdot)$.  First, we need two notations.  For
$a,b\in L$, let $k_a$ be the mapping from $L$ to $L$ that maps
constantly to $a$, and let $f_{a,b}$ be the mapping defined by
\begin{displaymath}
  f_{a,b}:L\rightarrow L \:, \quad x\mapsto \begin{cases}
    b\quad&\text{if }x\leq a \:,\\
    1&\text{otherwise} \:.
  \end{cases}
\end{displaymath}
The semiring $(S,\vee,\circ)$ may fulfill some of the following conditions:
\begin{align}
  &\forall a\in L\!\setminus\!\{1\}\ \forall b\in L:\ f_{a,b}\in
  S \:,\label{condition_f_ab}\\ 
  & \forall f\in S\ \exists a\in L\!\setminus\!\{1\}\ \exists b\in L:\
  f_{a,b}\leq f \:,\label{condition_f_ab_leq_f}\\
  &\forall a\in L:\ k_a\in S \:,\label{condition_k_a}\\
  &\forall f\in S\ \exists a\in L:\ k_a\leq f \:, 
  \label{condition_k_a_leq_f} \\
  &\forall a\in L\  \forall b\in L\!\setminus\!\{1\}\ \exists f\in S:\
  f(x)=b \text{ if }x\leq a,\ f(x)>b \text{ otherwise}\,.
  \label{condition_6}
\end{align}
If $\bL$ is a lattice then $(S,\vee,\circ)$ may also fulfill:
\begin{align}
  &\forall a\in L\!\setminus\!\{1\}:\ f_{a,0}\in S \:,\label{condition_e_a1}\\
  &\forall f\in S\ \exists a\in L\!\setminus\!\{1\}:\ f_{a,0}\leq f \:,
  \label{condition_e_a1_leq_f}\\
  &\forall a\in L\!\setminus\!\{0,1\}\ \forall b\in L\ \exists f\in S:\
  f(a)=b \:.\label{condition_fb=x}
\end{align}

We also need the following notations.  Let $\bL$ be a finite
semilattice and $\bK$ a finite lattice.  Then we denote:
\begin{align*}
  &\JM_1(\bL):=\{f\in \JM(\bL)\mid f(1)=1\} \:,\\
  &\Res_1(\bK):=\{f\in \Res(\bK)\mid f(1)=1\} \:,\\
  &\Res_0(\bK):=\{f\in \Res(\bK)\mid \forall x\in K:\ f(x)=0 
  \,\Rightarrow\, x=0\} \:.
\end{align*}


\subsection{$\infty_R$ is right but not left absorbing}

\begin{lemma}\label{lemma1}
  Let $\bL=(L,\leq)$ be a finite lattice, $a\in L\setminus\{1\}$ and
  $f\in \Res_1(\bL)$.  Then there exists an element $b\in
  L\setminus\{1\}$ such that $f_{b,0}=f_{a,0}\circ f$.
\end{lemma}

\begin{proof}
  Define $b:=\bigvee\{x\in L\mid f(x)\leq a\}$.  Then $f(x)\leq
  a\Leftrightarrow x\leq b$ holds for every $x\in L$ and it follows
  that $f_{b,0}=f_{a,0}\circ f$.  Because of $f(1)=1$ and $a<1$, it
  cannot hold that $b=1$.
\end{proof}

\begin{lemma}\label{lemma_e_a1_leq_f}
  Let $\bL=(L,\leq)$ be a finite lattice and $(S,\vee,\circ)$ a simple
  subsemiring of $(\Res_1(\bL),\vee,\circ)$ that fulfills
  (\ref{condition_e_a1}).  Then it also fulfills
  (\ref{condition_e_a1_leq_f}).
\end{lemma}

\begin{proof}
  Define the set $Z:=\{f\in S\mid \forall a\in L\setminus\{1\}:\
  f_{a,0}\not\le f\}$ and the equivalence relation $\sim$ on $S$ with
  the equivalence classes $S\setminus Z$ and $\{z\}$ for every $z\in
  Z$.  Let $f,g,h\in S$ with $f\sim g$ and $f\neq g$.  Consequently,
  $f$ and $g$ must be contained in $S\setminus Z$ and hence there
  exist $a,b\in L\setminus\{1\}$ with $f_{a,0}\leq f$ and $f_{b,0}\leq
  g$.  One can easily show that $f_{a,0}\leq f\vee h$, $f_{a,0}\leq
  h\circ f$, and $f_{c,0}\leq f\circ h$ for some $c\in
  L\setminus\{1\}$ holds for every $h\in S$, what yields $f\vee h$,
  $f\circ h$, $h\circ f\in S\setminus Z$.  Analogously, one can show
  $g\vee h$, $g\circ h$, $h\circ g\in S\setminus Z$ and it follows
  that $f\vee h\sim g\vee h$, $f\circ h\sim g\circ h$, $h\circ f\sim
  h\circ g$. Thus, $\sim$ is a congruence.  Since $\sim$ must be
  trivial and $S\setminus Z$ is a class with more than one element,
  $\sim\, = S\times S$ follows.  Hence, $Z=\varnothing$.
\end{proof}


Note in the following that $\End(M,+)=\JM(M,\leq)$ holds for a finite
idempotent semimodule $(M,+)$.

\begin{proposition}\label{prop_right_but_not_left_greatest_el}
  Let $\infty_R$ be right but not left absorbing.  Then there exists a
  finite lattice $\bL$ with more than two elements such that
  $(R,+,\cdot)$ is isomorphic to a subsemiring of
  $(\Res_1(\bL),\vee,\circ)$ which fulfills (\ref{condition_e_a1}),
  (\ref{condition_e_a1_leq_f}), and (\ref{condition_fb=x}).
\end{proposition}

\begin{proof}
  By Proposition~\ref{existence_of_irreducible_semimod}, there exists
  a finite idempotent irreducible $R$-semi\-module $(M,+)$, and
  $(R,+,\cdot)$ is isomorphic to the subsemiring $(T(R),+,\circ)$ of
  $(\JM(M,\leq),+,\circ)$, by Remark~\ref{remark_embed}.  From
  Proposition~\ref{prop_greatest_element} it follows that $(M,+)$ has
  a neutral element $0_M$, i.e., $(M,\leq)$ is a lattice, and
  $R0_M=\{0_M\}$ holds, as well as $R\infty_M=\{\infty_M\}$.  Thus,
  $(T(R),+,\cdot)$ is even a subsemiring of
  $({\Res_1(M,\leq)},+,\circ)$.  The lattice $(M,\le)$ must have more
  than two elements, because of $|R|>2$.  Now, (\ref{condition_e_a1})
  follows by Proposition~\ref{prop_r_a0}, (\ref{condition_e_a1_leq_f})
  by Lemma~\ref{lemma_e_a1_leq_f}, and (\ref{condition_fb=x}) by
  Proposition~\ref{prop_Ra=M}.
\end{proof}


\subsection{$\infty_R$ is left but not right absorbing}

To achieve a similar result for the case that $\infty_R$ is left but
not right absorbing, we need some preparation.

Let $\bL=(L,\leq)$ be a finite lattice with supremum $\vee$ and
infimum $\wedge$.  Then $\bL^d:=(L,\geq)$ is the dual lattice of $\bL$
with supremum $\vee^d:=\wedge$, infimum $\wedge^d:=\vee$, least
element $0_{\bL^d}=1_\bL$, and greatest element $1_{\bL^d}=0_\bL$.

For two mappings $f,g:S\rightarrow S$ on a set $S$, we define
$f\circ^d g := g\circ f$.

For a residuated mapping $f\in\Res(\bL)$ there exists a unique isotone
mapping $g:L\rightarrow L$ with $g\circ f\geq \id_L$ and $f\circ g\leq
\id_L$, which is called the \emph{residual} of~$f$.  We will denote
the residual of a residuated mapping $f$ by $f^+$ and we define
$\Rd(\bL) := \{f^+\mid f\in \Res(\bL)\}$ and $S^+ := \{f^+\mid f\in
S\}$ for $S\subseteq \Res(\bL)$.  For $f\in \Res(\bL)$ and $y\in L$ it
holds that $f^+(y) = \bigvee\{x\in L\mid f(x)\leq y\}$.  It further
holds that $\Rd(\bL) = \Res(\bL^d)$ and $(\Res(\bL),\vee,\circ) \cong
(\Rd(\bL),\wedge,\circ^d) = (\Res(\bL^d),\vee^d, \circ^d)$, where
$\Omega: f\mapsto f^+$ is an isomorphism between
$(\Res(\bL),\vee,\circ)$ and $(\Rd(\bL),\wedge,\circ^d)$
(see~\cite{janowitz}).

\begin{lemma}\label{S_plus}\sloppy
  Let $\bL=(L,\leq)$ be a lattice.  Then it holds that
  $(\Res_1(\bL),\vee,\circ) \cong
  (\Res_0(\bL^d),\vee^d,\circ^d)$.  Moreover, if $(S,\vee,\circ)$ is a
  subsemiring of $(\Res_1(\bL),\vee,\circ)$, then
  $(S^+,\vee^d,\circ^d)$ is a subsemiring of $(\Res_0(\bL^d),\vee^d,
  \circ^d)$ and it holds that ${(S,\vee,\circ)} \cong
  {(S^+,\vee^d,\circ^d)}$.
\end{lemma}

\begin{proof}
  Let $f\in \Res_1(\bL)$ and $y\in L$.  Since the set $\{x\in L\mid
  f(x)\leq y\}$ is closed under~$\bigvee$, we have that $f^+(y) =
  \bigvee\{x\in L\mid f(x)\leq y\} = 1_\bL$ implies $1_\bL=
  f(1_\bL)\leq y$, i.e., $y=1_\bL$.  Because of $1_\bL=0_{\bL^d}$, we
  have $f^+\in \Res_0(\bL^d)$.  Now let $g\in \Res(\bL)$ such that
  $g^+\in \Res_0(\bL^d)$, i.e., $g^+(y)=0_{\bL^d}=1_\bL$ implies
  $y=0_{\bL^d}=1_\bL$.  It follows that $g(1_\bL)=g^{++}(1_\bL) =
  \bigwedge\{y\in L\mid g^+(y)\geq 1_\bL\} = \bigwedge\{1_\bL\} =
  1_\bL$.  Thus, $g\in \Res_1(\bL)$.  Hence, $\Omega|_{\Res_1(\bL)}$
  is an isomorphism between $(\Res_1(\bL),\vee,\circ)$ and
  $(\Res_0(\bL^d),\vee^d, \circ^d)$, and for every subsemiring
  $(S,\vee,\circ)$ of $(\Res_1(\bL),\vee,\circ)$, we have
  $(S,\vee,\circ)\cong (\Omega(S),\vee^d,\circ^d)$ and
  $S^+=\Omega(S)\subseteq \Res_0(\bL^d)$.
\end{proof}

Let $\bL=(L,\leq)$ be a finite nontrivial lattice and define $L_{-} :=
L\setminus\{0_\bL\}$ and $\bL_{-} := (L_{-},\le \cap~ (L_{-}\times
L_{-}))$.  Then let $\Psi_\bL$ be the mapping defined by
\begin{displaymath}
  \Psi_\bL: \Res_0(\bL)\rightarrow \JM(\bL_{-}) \:, \quad 
  f \mapsto f|_{L_{-}} \:.
\end{displaymath}
The following lemma is easy to prove.

\begin{lemma}\label{Psi_S_plus}\sloppy
  Let $\bL = (L,\leq)$ be a finite nontrivial lattice.  Then
  $\Psi_\bL$ is an isomorphism between $(\Res_0(\bL),\vee,\circ)$ and
  $(\JM(\bL_{-}),\vee,\circ)$.  In particular,
  $(\Psi_\bL(S),\vee,\circ)$ is a subsemiring of
  $(\JM(\bL_{-}),\vee,\circ)$ and $(S,\vee,\circ) \cong
  (\Psi_\bL(S),\vee,\circ)$ holds for every subsemiring
  $(S,\vee,\circ)$ of $(\Res_0(\bL),\vee,\circ)$.
\end{lemma}

\begin{lemma}\label{lemma_derive_semiring}
  Let $\bK=(K,\leq)$ be a finite lattice and $\bL:=\bK^d$.  Moreover,
  let $(S,\vee,\circ)$ be a subsemiring of $(\Res_1(\bK),\vee,\circ)$
  which fulfills (\ref{condition_e_a1}), (\ref{condition_e_a1_leq_f}),
  and (\ref{condition_fb=x}).  Then $(\Psi_\bL(S^+),\vee^d,\circ)$ is
  a subsemiring of $(\JM(\bL_{-}),\vee^d,\circ)$ (where $\vee^d$
  refers to the supremum in $\bL=\bK^d$), which fulfills
  (\ref{condition_k_a}), (\ref{condition_k_a_leq_f}), and
  (\ref{condition_6}).
\end{lemma}

\begin{proof}\sloppy
  By Lemma~\ref{S_plus}, $(S^+,\vee^d,\circ^d)$ is a subsemiring of
  $(\Res_0(\bL),\vee^d,\circ^d)$ and therefore $(S^+,\vee^d,\circ)$ is
  a subsemiring of $(\Res_0(\bL),\vee^d,\circ)$.  By
  Lemma~\ref{Psi_S_plus}, $(\Psi_\bL(S^+),\vee^d,\circ)$ is a
  subsemiring of $(\JM(\bL_{-}),\vee^d,\circ)$.  By
  (\ref{condition_e_a1}), we have $f_{a,0_\bK}\in S$ for every $a\in
  K\setminus\{1_\bK\}$ and therefore $f_{a,0_\bK}^+\in S^+$ for every
  $a\in K\setminus\{1_\bK\}$, where
  \begin{displaymath}
    f_{a,0_\bK}^+(y) = \bigvee \{x\in K\mid f_{a,0_\bK}(x)\leq y\} = 
    \begin{cases}
      1_\bK \quad &\text{if }y=1_\bK \:,\\
      a & \text{otherwise} \:.
    \end{cases}
  \end{displaymath}
  Because of $L_{-} = L\setminus\{0_\bL\} = K\setminus\{1_\bK\}$, we
  get $\Psi_\bL(f_{a,0_\bK}^+) = f_{a,0_\bK}^+|_{K\setminus\{1_\bK\}}
  = k_a$.  Consequently, condition (\ref{condition_k_a}) is fulfilled.
  Now let $a\in K$ and $b\in K\setminus\{0_\bK,1_\bK\}$.  Then by
  (\ref{condition_fb=x}), there exists an $f\in S$ with $f(b)=a$ and
  $f^+(a) = \bigvee\{x\in K\mid f(x)\leq a\}\geq b$ holds.  Hence,
  $a\leq x$ implies $b\leq f^+(a)\leq f^+(x)$ for every $x\in K$.  Let
  $x\in K$ with $a\not\le x$ and assume $b\leq f^+(x)$.  It follows that
  $a = f(b)\leq f(f^+(x)) \leq \id(x) = x$, what is a contradiction,
  and we derive the equivalence $a\leq x\Leftrightarrow b\leq f^+(x)$
  for every $x\in K$.  If we use the order relation $\leq^d$ in
  $\bL=\bK^d$, then we have $f^+(x)\leq^d b$ if $x\leq^d a$ and
  $f^+(x)\not\le^d b$ otherwise.  Hence, $k_b\vee^d f^+(x)=b$ if
  $x\leq^d a$ and $k_b\vee^d f^+(x)>^d b $ otherwise.  The mapping
  $\Psi_\bL(k_b\vee^d f^+)$ is then the required mapping for $a$ and
  $b$ in condition (\ref{condition_6}).  Condition
  (\ref{condition_k_a_leq_f}) is satisfied by
  $(\Psi_\bL(S^+),\vee^d,\circ)$, because $(S,\vee,\circ)$ fulfills
  (\ref{condition_e_a1_leq_f}) and $(S,\vee)$ is isomorphic to
  $(\Psi_\bL(S^+),\vee^d)$.
\end{proof}

\begin{proposition}\label{proposition_first_part_left_but_not_right_greatest}
  Let $\infty_R$ be left but not right absorbing.  Then there exists a
  finite nontrivial semilattice $\bL$ such that $(R,+,\cdot)$ is
  isomorphic to a subsemiring of $(\JM(\bL),\vee,\circ)$ that fulfills
  (\ref{condition_k_a}), (\ref{condition_k_a_leq_f}), and
  (\ref{condition_6}).
\end{proposition}

\begin{proof}
  Define $r\star s := s\cdot r$ for every $r,s\in R$.  Then
  $(R,+,\star)$ is a finite simple additively idempotent semiring such
  that $\infty_R$ is right but not left absorbing.  By
  Proposition~\ref{prop_right_but_not_left_greatest_el}, there exists
  a finite lattice $\bK=(K,\leq)$ with $|K|\geq 3$ such that
  $(R,+,\star)$ is isomorphic to a subsemiring $(S,\vee,\circ)$ of
  $(\Res_1(\bK),\vee,\circ)$ that fulfills conditions
  (\ref{condition_e_a1}), (\ref{condition_e_a1_leq_f}), and
  (\ref{condition_fb=x}).  Let $\bL:=\bK^d$, then by
  Lemma~\ref{lemma_derive_semiring}, $(\Psi_\bL(S^+),\vee^d,\circ)$ is
  a subsemiring of $(\JM(\bL_{-}),\vee^d,\circ)$, which fulfills
  (\ref{condition_k_a}), (\ref{condition_k_a_leq_f}), and
  (\ref{condition_6}).  Clearly, $\bL_-$ is nontrivial.  Because of
  $(R,+,\star)\cong(S,\vee,\circ)\cong (S^+,\vee^d,\circ^d)$ by
  Lemma~\ref{S_plus}, it holds that $(R,+,\cdot)\cong
  (S^+,\vee^d,\circ)\cong (\Psi_\bL(S^+),\vee^d,\circ)$ by
  Lemma~\ref{Psi_S_plus}.
\end{proof}


\subsection{$\infty_R$ is absorbing}

The following proposition is \cite[Theorem 2.2]{jezek} for finite
semilattices.  Note that for a finite semilattice $\bL=(L,\leq)$ the
mappings $f_{a,b}$ with $a\in L\setminus\{1\}$ and $b\in L$ are
exactly the mappings of range at most two in $\JM_1(\bL)$.

\begin{proposition}\label{prop_f_ab_leq_f}
  Let $\bL=(L,\leq)$ be a finite nontrivial semilattice and
  $(S,\vee,\circ)$ a subsemiring of $(\JM_1(\bL),\vee,\circ)$ that
  fulfills (\ref{condition_f_ab}).  Then $(S,\vee,\circ)$ is simple
  iff it fulfills (\ref{condition_f_ab_leq_f}).
\end{proposition}

Recall that a finite idempotent semimodule $(M,+)$ satisfies~$(*)$ if
for its greatest element $\infty_M$ there exists $u\in M$ with
$\infty_M\ne u+x$ for all $x\in M\setminus\{\infty_M\}$.

\begin{proposition}\label{prop_irr_semimod_with_join-irr_greatest}
  Let $\infty_R$ be absorbing and let $(M,+)$ be a finite idempotent
  irreducible $R$-semimodule satisfying~$(*)$.  Then $(R,+,\cdot)$ is
  isomorphic to a subsemiring of $(\JM_1(M,\leq),+,\circ)$ that
  fulfills (\ref{condition_f_ab}) and (\ref{condition_f_ab_leq_f}).
\end{proposition}

\begin{proof}
  By Remark~\ref{remark_embed}, the semiring $(R,+,\cdot)$ is
  isomorphic to a subsemiring of $({\JM(M,\leq)},+,\circ)$, and
  because of $R\infty_M=\{\infty_M\}$ by
  Proposition~\ref{prop_greatest_element} even of
  $(\JM_1(M,\leq),+,\circ)$.  Since $|R|>2$ we must have that $|M|>2$
  as well.

  There exists $u\in M$ with $\infty_M \ne u+x$ for all $x\in
  M\setminus\{\infty_M\}$, and without loss of generality we may
  assume that $u\in\Min(M,\le)$.  By Proposition~\ref{prop_r_au} we
  have that $r_{a,u}\in R$ for every $a\in M\setminus\{\infty_M\}$.
  By Proposition~\ref{prop_Ra=M} and
  Proposition~\ref{prop_greatest_element}-\ref{infty_4}.\ it follows
  that $Ru=M$.  Hence, for all $b\in M$ there exists $s\in R$ such
  that $b=su$ and therefore $r_{a,b} = s\,r_{a,u}\in R$.  Thus,
  (\ref{condition_f_ab}) is fulfilled and (\ref{condition_f_ab_leq_f})
  follows by Proposition~\ref{prop_f_ab_leq_f}.
\end{proof}


\section{Subsemirings of $(\JM(\bL),\vee,\circ)$}%
\label{chapter_subsemirings}

In this section we consider the other direction, i.e., we start with a
semilattice~$\bL$ and show that certain subsemirings of
$(\JM(\bL),\vee,\circ)$ are simple.

\begin{proposition}\label{prop_second_part_right_but_not_left_greatest}
  Let $\bL=(L,\leq)$ be a finite lattice with more than two elements
  and let $(R,\vee,\circ)$ be a subsemiring of
  $(\Res_1(\bL),\vee,\circ)$, which fulfills (\ref{condition_e_a1}),
  (\ref{condition_e_a1_leq_f}), and (\ref{condition_fb=x}).  Then
  $(R,\vee,\circ)$ is a finite simple additively idempotent semiring
  and the greatest element is right but not left absorbing.
\end{proposition}

\begin{proof}
  It is clear that $(R,\vee,\circ)$ is a finite additively idempotent
  semiring; its greatest element is $\infty_R = f_{0,0}$.  It is easy
  to see that each element $f_{a,0}\in R$, where $a\in
  L\setminus\{1\}$, is right absorbing, hence in particular $f_{0,0}$
  is right and not left absorbing.

  To prove simplicity, let $\sim$ be a congruence on $(R,\vee,\circ)$,
  and suppose that $\sim\,\neq\id_R$, i.e., there are $f,g\in R$ such
  that $f\ne g$ and $f\sim g$.  Hence there exists $x\in L$ such that
  $f(x)\ne g(x)$, and we may assume that $f(x)\not\le g(x)=:a$.  Then
  we have $f_{a,0}\in R$, so that $f_{a,0}\circ f\sim f_{a,0}\circ g$,
  and there are $b,c\in L$ such that $f_{a,0}\circ f=f_{b,0}$ and
  $f_{a,0}\circ g=f_{c,0}$.  Furthermore, $f_{b,0}(x)=1$ and
  $f_{c,0}(x)=0$, so that $c\not\le b$.  We have shown that there are
  elements $b,c\in L$ with $c\not\le b$ such that $f_{b,0}\sim
  f_{c,0}$.

  Now we show that for all $z\in L\setminus\{0,1\}$ there exists $y\in
  L$, $y<z$ such that $f_{z,0}\sim f_{y,0}$.  So let $z\in
  L\setminus\{0,1\}$ and let $h\in R$ such that $h(z)=c$.  Considering
  $k:=h\vee f_{z,0}$ it is easy to see that $f_{c,0}\circ k =
  f_{z,0}$.  On the other hand there is $y\in L$ such that
  $f_{b,0}\circ k = f_{y,0}$, and it holds that $f_{y,0}(z)=1$.  From
  this and since $f_{z,0}\le f_{y,0}$ it follows $y<z$.  Furthermore,
  we have $f_{z,0} = f_{c,0}\circ k \sim f_{b,0}\circ k = f_{y,0}$ as
  desired.

  By applying the last paragraph repeatedly we see that for all $z\in
  L\setminus\{1\}$ it holds that $f_{z,0}\sim f_{0,0}$.  Now let $f\in
  R$ be arbitrary and let $z\in L\setminus\{1\}$ such that $f_{z,0}\le
  f$.  Then we get $f = f_{z,0}\vee f \sim f_{0,0}\vee f = f_{0,0}$.
  Hence $\sim\: = R\times R$, as desired.
\end{proof}

\begin{proposition}%
  \label{proposition_second_part_left_but_not_right_greatest}
  Let $\bL=(L,\leq)$ be a nontrivial finite semilattice and
  $(R,\vee,\circ)$ a subsemiring of $(\JM(\bL),\vee,\circ)$, which
  fulfills (\ref{condition_k_a}), (\ref{condition_k_a_leq_f}), and
  (\ref{condition_6}).  Then $(R,\vee,\circ)$ is a finite simple
  additively idempotent semiring and the greatest element is left but
  not right absorbing.
\end{proposition}

\begin{proof}
  It is clear that $(R,\vee,\circ)$ is a finite additively idempotent
  semiring; its greatest element is $\infty_R = k_1$.  Each element
  $k_a\in R$, where $a\in L$, is left absorbing, hence in particular
  $k_1$ is left but not right absorbing.

  To prove simplicity, let $\sim$ be a congruence on $(R,\vee,\circ)$,
  and suppose $\sim\,\neq\id_R$, i.e., there are $f,g\in R$ such that
  $f\ne g$ and $f\sim g$.  There exists $x\in L$ such that $f(x)\ne
  g(x)$, and we may assume $c:=f(x)\not\le g(x)=:b$.  Then we have
  $k_c=f\circ k_x\sim g\circ k_x=k_b$.

  Now for all $z\in L\setminus\{1\}$ there exists $y\in L$, $y>z$ such
  that $k_z\sim k_y$.  Indeed, let $h\in R$ such that $h(x)=z$ if
  $x\le b$, and $h(x)>z$ otherwise.  Then in particular $y:=h(c)>z$,
  and $k_y = h\circ k_c\sim h\circ k_b = k_z$.

  By applying the last paragraph repeatedly we see that $k_z\sim k_1$
  for all $z\in L$.  Now let $f\in R$ be arbitrary and let $z\in L$
  such that $k_z\le f$.  Then $f = k_z\vee f \sim k_1\vee f = k_1$.
  Consequently $\sim\:=R\times R$, as desired.
\end{proof}

\begin{proposition}\label{prop_semiring_with_absorbing_greatest}
  Let $\bL$ be a finite nontrivial semilattice and let
  $(R,\vee,\circ)$ be a subsemiring of $(\JM_1(\bL),\vee,\circ)$,
  which fulfills (\ref{condition_f_ab}) and
  (\ref{condition_f_ab_leq_f}). Then $(R,\vee,\circ)$ is a finite
  simple additively idempotent semiring with absorbing greatest
  element.
\end{proposition}

\begin{proof}
  Clearly, $(R,\vee,\circ)$ is finite and additively idempotent.  The
  simplicity holds by Proposition~\ref{prop_f_ab_leq_f}.  The greatest
  element is $f_{a,1}=k_1$, for arbitrary $a\in L\setminus\{1\}$,
  which is absorbing.
\end{proof}

\begin{proposition}\label{lemma_L_is_irr_semimod}
  Let $\bL=(L,\leq)$ be a finite nontrivial semilattice and
  $(R,\vee,\circ)$ a simple subsemiring of $(\JM_1(\bL),\vee,\circ)$,
  which fulfills (\ref{condition_f_ab}) and $|R|>2$. Then $(L,\vee)$
  is an irreducible $R$-semimodule.
\end{proposition}

\begin{proof}
  Clearly, $(L,\vee)$ is an $R$-semimodule, which is faithful and
  hence non-quasitrivial.  Let $(K,\vee)$ be an $R$-subsemimodule of
  $(L,\vee)$ with $|K|>1$.  Then there exists $a\in K$ with $a\neq
  1_\bL$ and it follows that $b=f_{a,b}(a)\in K$ for every $b\in
  L$. Thus $K=L$ and $(L,\vee)$ is consequently sub-irreducible.

  Let now $\sim$ be a semimodule congruence on $(L,\vee)$ with
  $\sim \:\ne \id_L$, i.e., there exist $a,b\in L$ with $a\neq b$ and
  $a\sim b$. Without loss of generality we can say $b\not\le a$. It
  follows $a\neq 1$. Choose $c\in L$ arbitrarily. Then
  $c=f_{a,c}(a)\sim f_{a,c}(b)=1$. Hence, $c\sim 1$ for every $c\in
  L$. Thus, $\sim \:= L\times L$ must hold. We conclude that $(L,\vee)$
  is quotient-irreducible.
\end{proof}


\section{Main results}%
\label{chapter_main_results}

Now we are ready to establish the characterization theorems for finite
simple additively idempotent semirings of all cases mentioned in the
introduction, except Case~\ref{case_remaining}.  The first theorem
states that the finite simple additively idempotent semirings with
greatest element that is neither left nor right absorbing are exactly
the finite simple additively idempotent semirings with zero.  It
follows from Proposition~\ref{prop_has_a_zero}; the second part of
the theorem is obvious.

\begin{theorem}\label{main_neither_left_nor_right}
  Let $(R,+,\cdot)$ be a finite simple additively idempotent semiring
  with $|R|>2$ and such that $\infty_R$ is neither left nor right
  absorbing. Then $(R,+,\cdot)$ is isomorphic to a semiring as in
  Theorem~\ref{theorem_zum}.  Conversely, every semiring in
  Theorem~\ref{theorem_zum} has a greatest element, which is neither
  left nor right absorbing.
\end{theorem}

We get the following theorem from
Proposition~\ref{prop_right_but_not_left_greatest_el} and
Proposition~\ref{prop_second_part_right_but_not_left_greatest}.

\begin{theorem}\label{main_right_but_not_left}
  Let $\bL$ be a finite lattice with more than two elements and let
  $(R,\vee,\circ)$ be a subsemiring of $(\Res_1(\bL),\vee,\circ)$,
  which fulfills (\ref{condition_e_a1}), (\ref{condition_e_a1_leq_f}),
  and (\ref{condition_fb=x}). Then $(R,\vee,\circ)$ is a finite simple
  additively idempotent semiring and the greatest element is right but
  not left absorbing.  Conversely, every finite simple additively
  idempotent semiring $(S,+,\cdot)$ with $|S|>2$ and with right but
  not left absorbing greatest element is isomorphic to such a
  semiring.
\end{theorem}

Proposition~\ref{proposition_first_part_left_but_not_right_greatest}
and
Proposition~\ref{proposition_second_part_left_but_not_right_greatest}
yield the following result.

\begin{theorem}\label{main_left_but_not_right}
  Let $\bL$ be a finite nontrivial semilattice and $(R,\vee,\circ)$ a
  subsemiring of $(\JM(\bL),\vee,\circ)$, which fulfills
  (\ref{condition_k_a}), (\ref{condition_k_a_leq_f}), and
  (\ref{condition_6}). Then $(R,\vee,\circ)$ is a finite simple
  additively idempotent semiring and the greatest element is left but
  not right absorbing. Conversely, every finite simple additively
  idempotent semiring $(S,+,\cdot)$ with $|S|>2$ and with left but not
  right absorbing greatest element is isomorphic to such a semiring.
\end{theorem}

The next theorem holds by
Proposition~\ref{prop_irr_semimod_with_join-irr_greatest},
Proposition~\ref{prop_semiring_with_absorbing_greatest}, and
Proposition~\ref{lemma_L_is_irr_semimod}.  Recall that we say that a
finite idempotent semimodule (or a finite semilattice) $(M,+)$
satisfies property $(*)$ if for its greatest element $\infty_M$ there
exists $u\in M$ with $\infty_M\ne u+x$ for all $x\in
M\setminus\{\infty_M\}$

\begin{theorem}\label{main_absorbing}
  Let $\bL$ be a nontrivial finite semilattice satisfying~$(*)$ and
  let $(R,\vee,\circ)$ be a subsemiring of $(\JM_1(\bL),\vee,\circ)$,
  which fulfills (\ref{condition_f_ab}) and
  (\ref{condition_f_ab_leq_f}).  Then $(R,\vee,\circ)$ is a finite
  simple additively idempotent semiring with absorbing greatest
  element and it possesses an idempotent irreducible $R$-semimodule
  satisfying~$(*)$.  Conversely, every finite simple additively
  idempotent semiring $(S,+,\cdot)$ with $|S|>2$, with absorbing
  greatest element, and which possesses an idempotent irreducible
  $S$-semimodule satisfying~$(*)$ is isomorphic to such a semiring.
\end{theorem}


\section{Isomorphic semirings}%
\label{chapter_isomorphic_semirings}

In this section we show that if we have two semirings as in
Theorem~\ref{main_right_but_not_left},
Theorem~\ref{main_left_but_not_right}, or 
Theorem~\ref{main_absorbing} that are isomorphic, then the
corresponding semilattices have to be isomorphic as well.  In
\cite{zum} the same was done for semirings as in
Theorem~\ref{theorem_zum} (Theorem~\ref{main_neither_left_nor_right}).

An \textit{order isomorphism} (resp.\ \textit{dual order isomorphism})
between two ordered sets $(P,\leq)$ and $(Q,\leq)$ is a surjective
mapping $f:P\rightarrow Q$ with $x\leq y\Leftrightarrow f(x)\leq f(y)$
(resp.\ $x\leq y\Leftrightarrow f(x)\geq f(y)$) for every $x,y\in
P$. Note that a (dual) oder isomorphism is necessarily bijective.  An
\textit{order automorphism} of $(P,\leq)$ is an order isomorphism from
$(P,\leq)$ to $(P,\leq)$.

\begin{lemma}\label{lemma_f_00_R}
  Let $\bL=(L,\leq)$ be a finite lattice and $(R,\vee,\circ)$ a
  subsemiring of $(\Res_1(\bL),\vee,\circ)$ that fulfills
  (\ref{condition_e_a1}).  Then
  \begin{displaymath}
    \Gamma: L\!\setminus\!\{1\}\rightarrow f_{0,0}\circ R := 
    \{f_{0,0}\circ f\mid f\in R\} \:, \quad a\mapsto f_{a,0}
  \end{displaymath}
  is a dual order isomorphism between $(L\!\setminus\!\{1\},\le)$ and
  $(f_{0,0}\circ R,\le)$.
\end{lemma}

\begin{proof}
  First we verify $f_{0,0}\circ R=\{f_{a,0}\mid a\in
  L\setminus\{1\}\}$.  The inclusion ``$\subseteq$'' holds by
  Lemma~\ref{lemma1}.  Now let $a\in L\setminus\{1\}$.  Then
  $f_{a,0}=f_{0,0}\circ f_{a,0}\in f_{0,0}\circ R$.  This proves the
  equality and it follows that $\Gamma$ is well-defined and
  surjective.  Because of $a\le b \Leftrightarrow f_{a,0}\ge f_{b,0}$
  for all $a,b\in L\setminus\{1\}$, $\Gamma$ is a dual order
  isomorphism.
\end{proof}

\begin{lemma}\label{lemma_R_k_1}
  Let $\bL=(L,\leq)$ be a finite semilattice and $(R,\vee,\circ)$ a
  subsemiring of $(\JM(\bL),\vee,\circ)$ that fulfills
  (\ref{condition_k_a}).  Then
  \begin{displaymath}
    \Lambda: L\rightarrow R\circ k_1:=\{f\circ k_1\mid f\in R\}
    \:, \quad a\mapsto k_a
  \end{displaymath}
  is an order isomorphism between $\bL$ and $(R\circ k_1,\le)$.
\end{lemma}

\begin{proof}
  First we verify $R\circ k_1 = \{k_a\mid a\in L\}$.  Let $f\in R$.
  Then $f\circ k_1=k_{f(1)}\in \{k_a\mid a\in L\}$.  Now let $a\in L$.
  Then $k_a=k_a\circ k_1\in R\circ k_1$.  This proves the equality and
  it follows that $\Lambda$ is well-defined and surjective.  Because
  of $a\leq b\Leftrightarrow k_a\leq k_b$ for all $a,b\in L$,
  $\Lambda$ is an order isomorphism.
\end{proof}

\begin{lemma}\label{lemma_R_f_ab}
  Let $\bL=(L,\leq)$ be a finite semilattice, $(R,\vee,\circ)$ a
  subsemiring of $(\JM_1(\bL),\vee,\circ)$ that fulfills
  (\ref{condition_f_ab}) and let $a,b\in L\setminus\{1\}$.  Then
  \begin{displaymath}
    \Phi: L\rightarrow R\circ f_{a,b}:=\{f\circ f_{a,b}\mid f\in R\}
    \:, \quad c\mapsto f_{a,c}
  \end{displaymath}
  is an order isomorphism between $\bL$ and $(R\circ f_{a,b},\le)$.
\end{lemma}

\begin{proof}
  First we verify $R\circ f_{a,b}=\{f_{a,c}\mid c\in L\}$.  Let $f\in
  R$.  Then $f\circ f_{a,b}=f_{a,f(b)}\in \{f_{a,c}\mid c\in L\}$.
  Now let $c\in L$.  Then $f_{a,c} = f_{b,c} \circ f_{a,b} \in R\circ
  f_{a,b}$.  This proves the equality and it follows that $\Phi$ is
  well-defined and surjective.  Since $c\leq d\Leftrightarrow
  f_{a,c}\leq f_{a,d}$ for all $c,d\in L$, $\Phi$ is an order
  isomorphism.
\end{proof}

\begin{proposition}
  Let $\bL_i=(L_i,\leq)$ be a finite lattice and $(R_i,\vee,\circ)$ a
  subsemiring of $(\Res_1(\bL_i),\vee,\circ)$ as in
  Theorem~\ref{main_right_but_not_left} for $i=1,2$. If
  $(R_1,\vee,\circ)$ and $(R_2,\vee,\circ)$ are isomorphic, then
  $\bL_1$ and $\bL_2$ are also isomorphic.
\end{proposition}

\begin{proof}
  Let $(R_1,\vee,\circ)$ and $(R_2,\vee,\circ)$ be isomorphic and let
  $\Omega:R_1\rightarrow R_2$ be an isomorphism.  Let $0_i:=0_{\bL_i}$
  for $i=1,2$.  Since $f_{0_i,0_i}$ is the greatest element in
  $(R_i,\leq)$, we have $\Omega(f_{0_1,0_1})=f_{0_2,0_2}$.  It follows
  that $\Omega(f_{0_1,0_1}\circ R_1) = \Omega(f_{0_1,0_1})\circ
  \Omega(R_1) = f_{0_2,0_2}\circ R_2$.  Hence, $(f_{0_1,0_1}\circ
  R_1,\leq)\cong (f_{0_2,0_2}\circ R_2,\leq)$.  With
  Lemma~\ref{lemma_f_00_R} we find that
  $(L_1\setminus\{1_{\bL_1}\},\leq)\cong (f_{0_1,0_1}\circ
  R_1,\geq)\cong (f_{0_2,0_2}\circ R_2,\geq)\cong
  (L_2\setminus\{1_{\bL_2}\},\leq)$.  It trivially follows that
  $\bL_1\cong \bL_2$.
\end{proof}

\begin{proposition}
  Let $\bL_i=(L_i,\leq)$ be a finite semilattice and
  $(R_i,\vee,\circ)$ a subsemiring of $(\JM(\bL_i),\vee,\circ)$ as in
  Theorem~\ref{main_left_but_not_right} for $i=1,2$.  If
  $(R_1,\vee,\circ)$ and $(R_2,\vee,\circ)$ are isomorphic, then
  $\bL_1$ and $\bL_2$ are also isomorphic.
\end{proposition}

\begin{proof}
  Let $(R_1,\vee,\circ)$ and $(R_2,\vee,\circ)$ be isomorphic and let
  $\Omega:R_1\rightarrow R_2$ be an isomorphism.  Let here
  $1_i:=1_{\bL_i}$ for $i=1,2$.  Since $k_{1_i}$ is the greatest
  element in $(R_i,\leq)$, we have $\Omega(k_{1_1})=k_{1_2}$.  It
  follows that $\Omega(R_1\circ k_{1_1}) =
  \Omega(R_1)\circ\Omega(k_{1_1}) = R_2\circ k_{1_2}$.  Hence,
  $(R_1\circ k_{1_1},\leq)\cong (R_2\circ k_{1_2},\leq)$.  From
  Lemma~\ref{lemma_R_k_1} follows that $\bL_1\cong (R_1\circ
  k_{1_1},\leq)\cong (R_2\circ k_{1_2},\leq)\cong \bL_2$.
\end{proof}

\begin{proposition}\label{prop_isomorphic}
  Let $\bL_i=(L_i,\leq)$ be a finite semilattice and
  $(R_i,\vee,\circ)$ a subsemiring of $(\JM_1(\bL_i),\vee,\circ)$ as
  in Theorem~\ref{main_absorbing} for $i=1,2$.  If $(R_1,\vee,\circ)$
  and $(R_2,\vee,\circ)$ are isomorphic, then $\bL_1$ and $\bL_2$ are
  also isomorphic.
\end{proposition}

An element $a$ in a finite semilattice $\bL$ is called \textit{coatom}
of $\bL$ if it is a lower neighbor of $1$.  With $\CoAt(\bL)$ we
denote the set of coatoms in $\bL$.

\begin{proof}
  Let $(R_1,\vee,\circ)$ and $(R_2,\vee,\circ)$ be isomorphic and let
  $\Omega:R_1\rightarrow R_2$ be an isomorphism.  One can easily show
  that $\CoAt(\JM_1(\bL_i),\leq) = \{f_{a,b}\mid a\in \Min(\bL_i),
  b\in \CoAt(\bL_i)\}$ holds.  Thus for $a\in \Min(\bL_1)$, $b\in
  \CoAt(\bL_1)$ there exist $a'\in \Min(\bL_2)$, $b'\in \CoAt(\bL_2)$
  with $\Omega(f_{a,b})=f_{a',b'}$.  We find that $\Omega(R_1\circ
  f_{a,b})=\Omega(R_1)\circ \Omega(f_{a,b})=R_2\circ f_{a',b'}$.
  Hence, $(R_1\circ f_{a,b},\leq)\cong (R_2\circ f_{a',b'},\leq)$.
  From Lemma~\ref{lemma_R_f_ab} follows that $\bL_1\cong(R_1\circ
  f_{a,b},\leq)\cong (R_2\circ f_{a',b'},\leq)\cong \bL_2$.
\end{proof}

We remark that, along similar lines as in this section, one can also
prove that for every semiring characterized in
Section~\ref{chapter_main_results} there exists up to isomorphism a
unique idempotent irreducible semimodule (with property ($*$), in the
case of Theorem~\ref{main_absorbing}).


\section{Neutral elements}%
\label{chapter_neutral_elements}


\subsection{Additively neutral element}

If the greatest element $1$ of a finite lattice is join-irreducible,
then we denote the unique lower neighbor of $1$ by $1_*$.

\begin{proposition}
  Let $\bL=(L,\leq)$ be a finite lattice and $(R,\vee,\circ)$ a
  semiring as in Theorem~\ref{main_right_but_not_left}.  Then
  $(R,\vee)$ has a neutral element iff $1$ is join-irreducible.  If
  the neutral element exists, then it is right but not left absorbing.
\end{proposition}

\begin{proof}
  If $1$ is join-irreducible, then $f_{1_*,0}$ is clearly a neutral
  element in $(R,\vee)$.  If $(R,\vee)$ has a neutral element $f_0$
  then it must fulfill $f_0(a)\leq f_{a,0}(a)=0$ for every $a\in
  L\setminus\{1\}$.  For all $a,b\in L\setminus\{1\}$ we have that
  $a\vee b \neq 1$ because of $f_0(a\vee b)=f_0(a)\vee f_0(b)=0$,
  i.e., $1$ is join-irreducible.

  The element $f_{1_*,0}$ is right absorbing, since $f(0)=0$ and $f(1)=1$
  for all $f\in R$.  But it is not left absorbing, since
  $f_{1_*,0}\circ f_{a,0}=f_{a,0}$ for all $a\in L\setminus\{1\}$.
\end{proof}

\begin{proposition}
  Let $\bL=(L,\leq)$ be a finite semilattice and $(R,\vee,\circ)$ a
  semiring as in Theorem~\ref{main_left_but_not_right}.  Then
  $(R,\vee)$ has a neutral element iff $\bL$ is a lattice.  If the
  neutral element exists, then it is left but not right absorbing.
\end{proposition}

\begin{proof}
  If $\bL$ is a lattice, then $k_0$ is clearly a neutral element in
  $(R,\vee)$.  If $(R,\vee)$ has a neutral element $f_0$, then it must
  fulfill $f_0(x)\leq k_a(x)=a$ for every $a,x\in L$.  Thus for all
  $x\in L$, $f_0(x)$ is the least element in $\bL$, i.e., $\bL$ is a
  lattice and it holds that $f_0=k_0$.  Clearly, $k_0$ is left
  absorbing, but it is not right absorbing because of $k_1\circ k_0 =
  k_1$.
\end{proof}

\begin{proposition}
  Let $\bL$ be a finite semilattice and $(R,\vee,\circ)$ a semiring as
  in Theorem~\ref{main_absorbing}.  Then $(R,\vee)$ has a neutral
  element iff $1$ is join-irreducible and $\bL$ is a lattice.  If the
  neutral element exists, then it is neither left nor right absorbing.
\end{proposition}

\begin{proof}
  If $1$ is join-irreducible and $\bL$ is a lattice, then $f_{1_*,0}$
  is a neutral element in $(R,\vee)$.  If $(R,\vee)$ has a neutral
  element $f_0$, then it must fulfill $f_0(x)\leq f_{x,a}(x)=a$ for
  every $a\in L$ and $x\in L\setminus\{1\}$.  Thus for $x\in
  L\setminus\{1\}$, $f_0(x)$ is the least element in $\bL$, i.e.,
  $\bL$ is a lattice and $f_0(x)=0$ holds.  Also, for all $a,b\in
  L\setminus\{1\}$ we have that $a\vee b \neq 1$ because of $f_0(a\vee
  b)=f_0(a)\vee f_0(b)=0$, i.e., $1$ is join-irreducible.  Since
  $f_{1_*,1}=k_1$ is absorbing, $f_0$ cannot be left or right
  absorbing.
\end{proof}

When considering finite simple additively idempotent semirings with an
additively neutral element, any finite idempotent irreducible
semimodule over such a semiring has a neutral element by
Proposition~\ref{prop_semirings_with_neutral_element}, and thus
satisfies $(*)$.  Hence, semirings of this kind with an absorbing
greatest element are already characterized by Theorem
\ref{main_absorbing}.

Therefore the classification of finite simple semirings with
additively neutral element is complete and can be summarized as in the
next theorem.

\begin{theorem}\label{classification_add_neutr}
  Let $(R,+,\cdot)$ be a finite  semiring with additively neutral element.
  Then $(R,+,\cdot)$ is simple iff one of the following holds:
  \begin{enumerate}
  \item $|R|\leq 2$,
  \item $(R,+,\cdot)\cong(\Mat_n(\mathbb{F}_q),+,\cdot)$ for some
    finite field $\mathbb{F}_q$ and some $n\geq 1$,
  \item $(R,+,\cdot)$ is a zero multiplication ring of prime order,
  \item $(R,+,\cdot)$ is isomorphic to a semiring as in Theorem
    \ref{theorem_zum},
  \item $(R,+,\cdot)$ is isomorphic to a semiring as in Theorem
    \ref{main_right_but_not_left}, where $1$ is join-irreducible,
  \item $(R,+,\cdot)$ is isomorphic to a semiring as in Theorem
    \ref{main_left_but_not_right}, where $\bL$ is a lattice,
  \item $(R,+,\cdot)$ is isomorphic to a semiring as in Theorem
    \ref{main_absorbing}, where $1$ is join-irreducible and $\bL$ is a
    lattice.
  \end{enumerate}
\end{theorem}


\subsection{Multiplicatively neutral element}

\begin{proposition}
  Let $\bL=(L,\leq)$ be a lattice and $(R,\vee,\circ)$ a semiring as
  in Theorem~\ref{main_right_but_not_left}.  Then $(R,\circ)$ has a
  neutral element iff $\id_L\in R$.  If $\id_L\in R$ then $1_\bL$ is
  join-irreducible.
\end{proposition}

\begin{proof}
  If $\id_L\in R$ then it is clearly a neutral element of
  $(R,\circ)$.  Let $(R,\circ)$ have a neutral element $e$ and let
  $x\in L$.  For $a\in L\setminus\{0,1\}$ there exists $f\in R$ with
  $f(a)=x$.  It follows that $e(x)=e(f(a))=(e\circ f)(a)=f(a)=x$,
  i.e., $\id_L=e\in R$.

  If $\id_L\in R$ then there exists $a\in L\setminus\{1\}$ with
  $f_{a,0}\leq \id_L$, i.e., $x\not\le a$ implies $x=1$, for every
  $x\in L$.  Hence, $a$ is the unique lower neighbor of $1$, i.e., $1$
  is join-irreducible.
\end{proof}

\begin{proposition}
  Let $\bL=(L,\leq)$ be a semilattice and $(R,\vee,\circ)$ a semiring
  as in Theorem~\ref{main_left_but_not_right}.  Then $(R,\circ)$ has a
  neutral element iff $\id_L\in R$.  If $\id_L\in R$ then $\bL$ is a
  lattice.
\end{proposition}

\begin{proof}
  If $\id_L\in R$ then it is clearly a neutral element of
  $(R,\circ)$.  If $(R,\circ)$ has a neutral element $e$ then
  $e(x)=e(k_x(x))=(e\circ k_x)(x)=k_x(x)=x$ for every $x\in L$,
  i.e., $\id_L=e\in R$.

  If $\id_L\in R$ then there exists $a\in L$ with $k_a\leq \id_L$.
  Thus $a=k_a(x)\leq \id_L(x)=x$ for all $x\in L$.  Hence, $a$ is the
  least element in $\bL$, i.e., $\bL$ is a lattice.
\end{proof}

\begin{proposition}
  Let $\bL=(L,\leq)$ be a semilattice and $(R,\vee,\circ)$ a semiring
  as in Theorem~\ref{main_absorbing}.  Then $(R,\circ)$ has a neutral
  element iff $\id_L\in R$. If $\id_L\in R$ then $1_{\bL}$ is
  join-irreducible and $\bL$ is a lattice.
\end{proposition}

\begin{proof}
  If $\id_L\in R$ then it is clearly a neutral element in $(R,\circ)$.
  Let $(R,\circ)$ have a neutral element $e$ and let $x\in L$.  For
  $a\in L\setminus\{1\}$, the equality $e(x)=e(f_{a,x}(a))=(e\circ
  f_{a,x})(a)=f_{a,x}(a)=x$ holds, i.e., $\id_L=e\in R$.  
  
  If $\id_L\in R$ then there exists $a\in L\setminus\{1\}$ and $b\in
  L$ with $f_{a,b}\le \id_L$.  Thus $x\not\le a$ implies $x=1$, and
  $x\le a$ implies $b\le x$, for every $x\in L$.  Hence, $a$ is the
  unique lower neighbor of $1$, i.e., $1$ is join-irreducible.  Also,
  it follows that $b\le x$ for any $x\ne 1$, so that $b$ is the least
  element and $\bL$ is a lattice.
\end{proof}

>From the results in this section it also follows that the existence of
a multiplicatively neutral element implies the existence of an
additively neutral element, for all semirings in
Theorem~\ref{main_right_but_not_left},
Theorem~\ref{main_left_but_not_right}, and
Theorem~\ref{main_absorbing}.


\section{The remaining case}%
\label{chapter_remaining_case}

The semirings that elude our characterizisation theorems are the
finite simple additively idempotent semirings with absorbing greatest
element, which possess a finite idempotent irreducible semimodule $M$
\emph{without} property $(*)$, so that for all $u\in M$ there is $x\in
M\setminus\{\infty_M\}$ such that $\infty_M = u+x$.  For this case we
have a construction of semirings of join-morphisms of semilattices.
In fact, we conjecture that this construction covers these semirings.
We need some preparation for it.

\begin{defn}
  Let $\bL=(L,\le)$ and $\bK=(K,\le)$ be finite semilattices and let
  $A:=\{(x,y)\in L\times K\mid x=1_{\bL} \text{ or } y=1_\bK\}$.  Then
  define
  \[ L\boxtimes K :=  L \times K  \: / \: 
  (\id_{L\times K}\cup \, A\times A) 
  \quad \text{and} \quad 
  \bL\boxtimes \bK := (L\boxtimes K \,,\, \le) \:,  \]
  where $\{(a,b)\}\le A$ and $\{(a,b)\}\le \{(c,d)\}$ iff $a\le c$ and
  $b\le d$, for all $\{(a,b)\}$, $\{(c,d)\}\in L\boxtimes
  K \,\setminus\, \{A\}$.
\end{defn}

Note that every equivalence class in $L\boxtimes K$, except $A$, has
just one element, i.e., $L\boxtimes K = \{A\} \cup \big\{\{(a,b)\}\mid
a\in L\setminus\{1_\bL\}, b\in K\setminus \{1_\bK\}\big\}$.  See
Figure~\ref{figure_amalgam} for an example.

\begin{figure}
  \begin{center}
    \begin{picture}(300,100)
      \put(0,30){%
        \begin{diagram}{20}{20}
          \Node{0}{0}{0}
          \Node{1}{20}{0}
          \Node{2}{10}{20}
          \Edge{0}{2}
          \Edge{1}{2}
        \end{diagram}
      }
      \put(30,40){$\times$}
      \put(50,20){%
        \begin{diagram}{5}{30}
          \Node{0}{0}{0}
          \Node{1}{0}{20}
          \Node{2}{0}{40}
          \Edge{0}{1}
          \Edge{1}{2}
        \end{diagram}
      }
      \put(60,40){$=$}
      \put(80,20){%
        \begin{diagram}{60}{40}
          \Node{0}{0}{0}
          \Node{1}{0}{20}
          \Node{2}{0}{40}
          \Node{3}{20}{20}
          \Node{4}{20}{40}
          \Node{5}{20}{60}
          \Node{6}{40}{0}
          \Node{7}{40}{20}
          \Node{8}{40}{40}
          \Edge{0}{1}
          \Edge{1}{2}
          \Edge{3}{4}
          \Edge{4}{5}
          \Edge{6}{7}
          \Edge{7}{8}
          \Edge{0}{3}
          \Edge{1}{4}
          \Edge{2}{5}
          \Edge{6}{3}
          \Edge{7}{4}
          \Edge{8}{5}
        \end{diagram}
      }  
      \put(10,0){$\bL$}
      \put(45,0){$\bK$}
      \put(90,0){$\bL\times \bK$}
      \put(80,60){\ColorNode{black}}
      \put(120,60){\ColorNode{black}}
      \put(100,40){\ColorNode{black}}
      \put(100,60){\ColorNode{black}}
      \put(100,80){\ColorNode{black}}
      \put(180,30){%
        \begin{diagram}{20}{20}
          \Node{0}{0}{0}
          \Node{1}{20}{0}
          \Node{2}{10}{20}
          \Edge{0}{2}
          \Edge{1}{2}
        \end{diagram}
      }
      \put(210,40){$\boxtimes$}
      \put(230,20){%
        \begin{diagram}{5}{30}
          \Node{0}{0}{0}
          \Node{1}{0}{20}
          \Node{2}{0}{40}
          \Edge{0}{1}
          \Edge{1}{2}
        \end{diagram}
      }
      \put(240,40){$=$}
      \put(260,20){%
        \begin{diagram}{5}{30}
          \Node{0}{0}{0}
          \Node{1}{0}{20}
          \Node{2}{40}{0}
          \Node{3}{40}{20}
          \Node{4}{20}{40}
          \Edge{0}{1}
          \Edge{1}{4}
          \Edge{2}{3}
          \Edge{3}{4}
          \NoDots
          \centerObjbox{4}{0}{-10}{A}
        \end{diagram}
      }
      \put(190,0){$\bL$}
      \put(225,0){$\bK$}
      \put(270,0){$\bL\boxtimes\bK$}
    \end{picture}             
  \end{center}
  \caption{Left: The direct product of two semilattices. The black
    elements are the elements of the set $A$. Right: The product
    $\bL\boxtimes\bK$ of the same semilattices.}
  \label{figure_amalgam}
\end{figure}
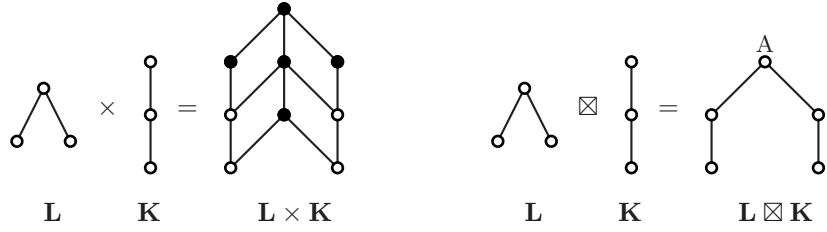

\begin{defn}
  Let $\bL=(L,\le)$ and $\bK=(K,\le)$ be finite semilattices, and let
  $f\in \JM_1(\bL)$ and $g\in \JM_1(\bK)$.  Then let $f\boxtimes g$ be
  the mapping in $\JM_1(\bL\boxtimes \bK)$ defined by
  \[ ( f\boxtimes g ) \big( [x,y] \big) = [ f(x), g(y) ] \]
  for every $(x,y)\in L\times K$, where $[x,y]$ denotes the class of
  $(x,y)$ in $L\boxtimes K$.
\end{defn}

Since $f\in \JM_1(\bL)$ and $g\in \JM_1(\bK)$ the map $f\boxtimes g$
is clearly well-defined.  Note that for $f_1,f_2\in \JM_1(\bL)$ and
$g_1,g_2\in \JM(\bK)$ the rules $(f_1\boxtimes g_1)\vee (f_2\boxtimes
g_2) = (f_1\vee f_2)\boxtimes (g_1\vee g_2)$ and $(f_1\boxtimes
g_1)\circ (f_2\boxtimes g_2) = (f_1\circ f_2)\boxtimes (g_1\circ g_2)$
apply.

With $\Aut(\bK)$ we denote the set of the automorphisms of a
semilattice $\bK$.  We consider in particular the case where
$\bK=(K,\le)$ is the semilattice
$(K=\{1,\dots,n\}\,\cupdot\,\{\infty\}, \le)$ where $\le \ :=\id_K
\cup\, ({K\times \{\infty\}})$, for some $n\in\mathbb{N}$; that is,
different elements are comparable only if one equals $\infty$.  In
this case, $\Aut(\bK)$ consists of all bijective maps $f:L\to L$ such
that $f(\infty)=\infty$, and thus the group $(\Aut(\bK),\circ)$ is
isomorphic to the symmetric group $\bS(K\setminus\{\infty\})$.  Any
subgroup $(S,\circ)$ of $(\Aut(\bK),\circ)$ acts in this sense
faithfully on the set $K\setminus\{\infty\} = \{1,\dots,n\}$.

\begin{construction}\label{construction}
  Let $\bL=(L,\le)$ be a semilattice and let $\bK:=(K,\le)$ be the
  semilattice, where $K=\{1,\dots,n\}\,\cupdot\,\{\infty\}$, $n\in
  \mathbb{N}$ and $\le \ :=\id_K \cup\, ({K\times \{\infty\}})$.
  Further let $(S,\circ)$ be a subgroup of $(\Aut(\bK),\circ)$ with
  $f\vee g=k_1$ for every $f,g\in S$ with $f\ne g$, let
  $\bar{S}:=S\cup \{k_1\}$, and let $(R,\vee,\circ)$ be a subsemiring
  of $(\JM_1( {\bL\boxtimes \bK}),\vee,\circ)$ with
  \begin{align}
    &\forall \varphi\in R\ \exists f\in \JM_1(\bL)\ \exists g\in \bar{S}:\
    \varphi = f\boxtimes g \:, \label{condition_a}\\
    &\forall a\in L\!\setminus\!\{1_\bL\}\ \forall b\in L\ 
    \forall g\in \bar{S}:\ f_{a,b}\boxtimes g\in R \:, \label{condition_b}\\
    &\forall \varphi\in R\ \exists a\in L\!\setminus\!\{1_\bL\}\ 
    \exists b\in L\ \exists g\in \bar{S}:\ f_{a,b}\boxtimes g\le \varphi \:. 
    \label{condition_c}
  \end{align}
\end{construction}

If $|K|=2$ then $\bL\boxtimes \bK\cong\bL$ and $(R,\vee,\circ)$
corresponds to a subsemiring $(S,\vee,\circ)$ of
$(\JM_1(\bL),\vee,\circ)$, which fulfills (\ref{condition_f_ab}) and
(\ref{condition_f_ab_leq_f}).  If $\bL$ does not satisfy $(*)$ then
$(R,\vee,\circ)$ possesses also a finite irreducible idempotent
$R$-semimodule which does not satisfy $(*)$, namely $(L,\vee)$ (see
Proposition~\ref{lemma_L_is_irr_semimod}).

If $|L|=2$ then $\bL\boxtimes \bK\cong\bK$ and $(R,\vee,\circ)$
belongs to a class of finite simple semirings with absorbing greatest
element, which are also known.  These semirings have been presented in
the case of commutative semirings in \cite{bashir_2001} and for not
necessarily commutative semirings in \cite{monico}: Let $(G,\cdot)$ be
a finite group and define $V(G):=G \,\cupdot\, \{\infty\}$.  Extend
the multiplication of $G$ to $V(G)$ by the rule $x\infty=\infty
x=\infty$ for every $x\in V(G)$ and define the addition on $V(G)$ by
$x+x=x$ and $x+y=\infty$ for every $x,y\in V(G)$ with $x\ne y$.  Then
$(V(G),+,\cdot)$ is a finite simple additively idempotent semiring
with absorbing greatest element and $(V(G),+)$ is a finite irreducible
idempotent semimodule without property $(*)$ if $|G|>1$.

Construction~\ref{construction} is a combination of those two types of
semirings.  As shown in the next proposition, these semirings are also
simple.


\begin{proposition}
  Let everything as in Construction~\ref{construction}. Then
  $(R,\vee,\circ)$ is a finite simple additively idempotent semiring
  with absorbing greatest element.
\end{proposition}

\begin{proof}
  Clearly, $(R,\vee,\circ)$ is a finite additively idempotent
  semiring.  Its greatest element is $k_{1_{\bL\boxtimes\bK}}$, which
  is obviously absorbing.  Let $\sim$ be a congruence on
  $(R,\vee,\circ)$ with $\sim \:\ne \id_R$, i.e., there exist
  $\varphi,\gamma\in R$ with $\varphi\ne \gamma$ and $\varphi\sim
  \gamma$.  By (\ref{condition_a}) there exist $\varphi_1,\gamma_1\in
  \JM_1(\bL)$, $\varphi_2,\gamma_2\in \bar{S}$ with $\varphi =
  \varphi_1\boxtimes \varphi_2$ and
  $\gamma=\gamma_1\boxtimes\gamma_2$.  Without loss of generality we
  can assume $\varphi\ngeq\gamma$.  It follows that $\varphi\ne
  k_{1_{\bL\boxtimes\bK}}$.  Choose $\lambda\in
  R\setminus\{k_{1_{\bL\boxtimes\bK}}\}$ arbitrarily.  We will show
  that $\lambda\sim k_{1_{\bL\boxtimes\bK}}$ holds.  From this it
  follows that $\sim \:= R\times R$ and therefore the simplicity.

  Again there exist $\lambda_1\in\JM_1(\bL)$, $\lambda_2\in \bar{S}$
  with $\lambda=\lambda_1\boxtimes \lambda_2$.  By (\ref{condition_c})
  there exists $a\in L\setminus\{1_\bL\}$, $b\in L$ and $g\in\bar{S}$
  such that $f_{a,b}\boxtimes g\le \lambda_1\boxtimes\lambda_2$.  We
  have $\lambda_2\ne k_{1_{\bK}}$ and thus $\lambda_2(y)\ne 1_{\bK}$
  for some $y\in K$.  For all $x\in L$ it follows $[f_{a,b}(x), g(y)]
  \le [\lambda_1(x), \lambda_2(y)]$, so that $f_{a,b}(x)\le
  \lambda_1(x)$; hence $f_{a,b}\le \lambda_1$.  Because of $\varphi\ne
  k_{1_{\bL\boxtimes\bK}}$ it holds that $\varphi_1\ne k_{1_\bL}$ and
  thus there exists $x\in L$ with $c:=\varphi_1(x)\ne 1_\bL$. It
  follows that $f_{c,b}\circ \varphi_1\circ f_{a,x}\vee \lambda_1=
  f_{a,b}\vee \lambda_1 = \lambda_1$.  It also must hold that
  $\varphi_2,\lambda_2\ne k_{1_\bK}$, i.e., $\varphi_2,\lambda_2\in
  S$. Since $(S,\circ)$ is a group there exists $v\in S$ with
  $\varphi_2\circ v=\lambda_2$.  We make a distinction of cases.

  Case 1: $\gamma = k_{1_{\bL\boxtimes\bK}}$.  It holds that
  $(f_{c,b}\boxtimes \id_K)\circ (\varphi_1\boxtimes \varphi_2) \circ
  (f_{a,x}\boxtimes v)\vee (\lambda_1\boxtimes \lambda_2) =
  (f_{c,b}\circ \varphi_1\circ f_{a,x} \vee \lambda_1)\boxtimes
  (\id_K\circ\varphi_2\circ v \vee \lambda_2) = (\lambda_1\boxtimes
  \lambda_2)=\lambda$ and $(f_{c,b}\boxtimes \id_K)\circ
  (\gamma_1\boxtimes \gamma_2) \circ (f_{a,x}\boxtimes v)\vee
  (\lambda_1\boxtimes \lambda_2) = k_{1_{\bL\boxtimes\bK}} \vee
  (\lambda_1\boxtimes \lambda_2) = k_{1_{\bL\boxtimes\bK}}$ and
  because of $\varphi\sim \gamma$ it follows that $\lambda\sim
  k_{1_{\bL\boxtimes\bK}}$.

  Case 2: $\gamma\ne k_{1_{\bL\boxtimes\bK}}$ and $\varphi_1 =
  \gamma_1$.  It must hold that $\varphi_2\ne\gamma_2$ and it follows
  that $\lambda_2=\varphi_2\circ v\ne \gamma_2\circ v$, i.e.,
  $\lambda_2\vee \gamma_2\circ v = k_{1_\bK}$.  As in the previous
  case one can show the equality $(f_{c,b}\boxtimes \id_K)\circ
  (\varphi_1\boxtimes \varphi_2) \circ (f_{a,x}\boxtimes v)\vee
  (\lambda_1\boxtimes \lambda_2)=\lambda$.  Additionally it holds in
  this case that $(f_{c,b}\boxtimes \id_K)\circ (\gamma_1\boxtimes
  \gamma_2) \circ (f_{a,x}\boxtimes v)\vee (\lambda_1\boxtimes
  \lambda_2) = (f_{c,b}\circ \gamma_1 \circ f_{a,x}\vee
  \lambda_1)\boxtimes (\id_K\circ\gamma_2\circ v\vee
  \lambda_2)=(f_{c,b}\circ \gamma_1 \circ f_{a,x}\vee
  \lambda_1)\boxtimes k_{1_\bK}=k_{1_{\bL\boxtimes\bK}}$ and we find
  again that $\lambda\sim k_{1_{\bL\boxtimes\bK}}$.

  Case 3: $\gamma\ne k_{1_{\bL\boxtimes\bK}}$ and $\varphi_1 \nleq
  \gamma_1$.  There exists $y\in L$ with $\varphi_1(y)\nleq
  \gamma_1(y)=:d$.  It also holds that $\gamma_2\ne k_{1_\bK}$, i.e.,
  $\gamma_2\in S$.  Consequently there exists $w\in S$ with
  $\gamma_2\circ w= \lambda_2$.  It follows that $(f_{d,b}\boxtimes
  \id_K)\circ (\varphi_1\boxtimes \varphi_2)\circ( f_{a,y}\boxtimes
  w)\vee (\lambda_1\boxtimes \lambda_2)=(f_{d,b}\circ \varphi_1\circ
  f_{a,y} \vee \lambda_1)\boxtimes (\id_K\circ \varphi_2\circ w \vee
  \lambda_2)=(k_{1_\bL}\vee \lambda_1)\boxtimes (\varphi_2\circ w \vee
  \lambda_2) = k_{1_{\bL\boxtimes\bK}}$.  Further it holds that
  $(f_{d,b}\boxtimes \id_K)\circ (\gamma_1\boxtimes \gamma_2)\circ(
  f_{a,y}\boxtimes w)\vee (\lambda_1\boxtimes \lambda_2)=(f_{d,b}\circ
  \gamma_1\circ f_{a,y} \vee \lambda_1)\boxtimes (\id_K\circ
  \gamma_2\circ w \vee \lambda_2)=(f_{a,b}\vee \lambda_1)\boxtimes
  (\lambda_2\vee \lambda_2) =\lambda$ and it holds again that
  $\lambda\sim k_{1_{\bL\boxtimes\bK}}$.

  Case 4: $\gamma\ne k_{1_{\bL\boxtimes\bK}}$ and $\varphi_1 \ngeq
  \gamma_1$. In this case there exists $z\in L$ with
  $e:=\varphi_1(z)\ngeq \gamma_1(z)$.  Analogously to the previous
  case one can show that $(f_{e,b}\boxtimes \id_K)\circ
  (\varphi_1\boxtimes \varphi_2)\circ( f_{a,z}\boxtimes v)\vee
  (\lambda_1\boxtimes \lambda_2) =\lambda$ and $(f_{e,b}\boxtimes
  \id_K)\circ (\gamma_1\boxtimes \gamma_2)\circ( f_{a,z}\boxtimes
  v)\vee (\lambda_1\boxtimes \lambda_2)= k_{1_{\bL\boxtimes\bK}}$
  holds and we find again $\lambda\sim k_{1_{\bL\boxtimes\bK}}$.
\end{proof}

\begin{proposition}
  Let everything as in Construction~\ref{construction}.  Additionally,
  let $|S|=n$, and let $n>1$ or let $\bL$ without property $(*)$.
  Then ${(L\boxtimes K,\vee)}$ is a finite idempotent irreducible
  $R$-semimodule without property $(*)$.
\end{proposition}

\begin{proof}
  It is easy to see that $(L\boxtimes K,\vee)$ is a finite idempotent
  $R$-semimodule without property $(*)$.  Further it is an
  $R$-nonidentity semimodule and it fulfills $|R\,(L\boxtimes K)|>1$.

  Considering the action of the group $(S,\circ)$ on the set
  $K\setminus\{\infty\} = \{1,\dots,n\}$ it follows from the
  conditions in Construction~\ref{construction} that for every
  $x\in\{1,\dots,n\}$ the orbit map $S\to \{1,\dots,n\}$, $g\mapsto
  g(x)$ is injective; now, since $|S|=n$, this map is even bijective.

  Let $(M,\vee)$ be an $R$-subsemimodule of $(L\boxtimes K,\vee)$ with
  $|M|>1$, i.e., there exist $[a,b]\in M$ with $[a,b]\ne A$.  Hence
  $a\ne 1_\bL$ and $b\ne 1_\bK$.  Choose $[c,d]\in L\boxtimes K$
  arbitrarily.  Since the orbit map $g\mapsto g(b)$ is bijective it
  follows that there exits $g\in S$ with $g(b)=d$.  It follows that
  $(f_{a,c}\boxtimes g)([a,b])=[c,d]$.  Thus, $M = R\,[a,b] =
  L\boxtimes K$ and $(L\boxtimes K,\vee)$ is consequently
  sub-irreducible.

  Let $\sim$ be a semimodule congruence on $(L\boxtimes K,\vee)$ with
  $\sim \:\ne \id$, i.e., there exist $[a,b],[c,d]\in L\boxtimes K$
  with $[a,b]\sim[c,d]$ and $[a,b]\ne[c,d]$.  Let $e\in L$, $f\in K$.
  We will show that $[e,f]\sim A $ holds.  From this it follows
  that $\sim \ = L\boxtimes K \,\times\, L\boxtimes K$, i.e.,
  $(L\boxtimes K,\vee)$ is quotient-irreducible.

  If $[a,b] = A$ then $[c,d]\ne A$, i.e., $d\ne 1_\bK$.  Hence, there
  exists $g\in S$ with $g(d)=f$ and it follows that $A =
  (f_{c,e}\boxtimes g)(A) \sim (f_{c,e}\boxtimes g)([c,d]) = [e,f]$.
  The case $[c,d]=A$ works analogously.  So from now on we can
  consider the case that $[a,b],[c,d]\ne A$.  If $a=c$ then it holds
  that $b\ne d$ and it follows that $[a,b]=[a,b]\vee [a,b]\sim
  [c,d]\vee [a,b]=[a,1_\bK]=A$.  We find $A\sim [e,f]$ as before.  Now
  consider the case $a\ngeq c$.  There exists $h\in S$ with $h(b)=f$
  and it follows that $[e,f] = (f_{a,e}\boxtimes h)([a,b]) \sim
  (f_{a,e}\boxtimes h)([c,d])=[1_\bL,h(d)]=A$.  The case $a\nleq c$
  works analogously. Hence $(L\boxtimes K,\vee)$ is irreducible.
\end{proof}

\begin{corollary}\label{corollary_conjecture}
  Let everything as in Construction~\ref{construction}.  Additionally,
  let ${|S|=n}$, and let $n>1$ or let $\bL$ be without property $(*)$.
  Then $(R,\vee,\circ)$ is a finite simple additively idempotent
  semiring with absorbing greatest element, which possesses a finite
  idempotent irreducible semimodule without property $(*)$.
\end{corollary}

We computed all finite simple additively idempotent semirings with
cardinality at most $10$.  From these semirings, every finite simple
additively idempotent semirings with absorbing greatest element, which
possesses a finite idempotent irreducible semimodule not satisfying
$(*)$, is isomorphic to a semiring in
Corollary~\ref{corollary_conjecture}.  For this reason we have the
following conjecture.

\begin{conjecture}
  Let $(R,\vee,\circ)$ be a finite simple additively idempotent
  semiring with absorbing greatest element, which possesses a finite
  idempotent irreducible semimodule without property $(*)$.  Then
  $(R,\vee,\circ)$ is isomorphic to a semiring in
  Corollary~\ref{corollary_conjecture}.
\end{conjecture}


\section{Examples}\label{chapter_example}

\subsection{$\infty_R$ is right but not left absorbing}

Let $\bL=(\{0,1,2\},\leq)$ be the total order with $3$ elements. The
following semiring, consisting of the mappings $a,b,c\in \Res_1(\bL)$,
is the unique finite simple additively idempotent semiring with right
but not left absorbing greatest element, induced by $\bL$:

\begin{displaymath}
  \begin{array}{c|ccc}
    x & 0 & 1 & 2 \\
    \hline	
    a(x) & 0 & 0 & 2 \\
    b(x) & 0 & 1 & 2 \\
    c(x) & 0 & 2 & 2
  \end{array}\qquad
  \begin{array}{c|ccc}
    \vee & a & b & c \\
    \hline
    a &  a & b & c \\
    b &  b & b & c \\
    c &  c & c & c
  \end{array}\qquad
  \begin{array}{c|ccc}
    \circ & a & b & c \\
    \hline
    a &  a & a & c \\
    b &  a & b & c \\
    c &  a & c & c
  \end{array}
\end{displaymath}

The following semirings are all finite simple additively idempotent
semirings with right but not left absorbing greatest element, induced
by $(\{0,1,2,3\},\leq)$:
\begin{gather*}
  \begin{array}{c}
    R_{7,1} \smallskip\\
    \begin{array}{c|cccc}
      x & 0 & 1 & 2 & 3 \\
      \hline	
      a(x) & 0 & 0 & 0 & 3 \\
      b(x) & 0 & 0 & 1 & 3 \\
      c(x) & 0 & 0 & 2 & 3 \\
      d(x) & 0 & 0 & 3 & 3 \\
      e(x) & 0 & 1 & 3 & 3 \\
      f(x) & 0 & 2 & 3 & 3 \\
      g(x) & 0 & 3 & 3 & 3
    \end{array} \\
  \end{array}
  \qquad
  \begin{array}{c}
    R_{7,2} \smallskip\\
    \begin{array}{c|cccc}
      x & 0 & 1 & 2 & 3 \\
      \hline	
      a(x) & 0 & 0 & 0 & 3 \\
      b(x) & 0 & 0 & 3 & 3 \\
      c(x) & 0 & 1 & 1 & 3 \\
      d(x) & 0 & 1 & 3 & 3 \\
      e(x) & 0 & 2 & 2 & 3 \\
      f(x) & 0 & 2 & 3 & 3 \\
      g(x) & 0 & 3 & 3 & 3
    \end{array} \\
  \end{array}\\
  \begin{array}{c}
    \\
    R_{8,1} \smallskip \\
    \begin{array}{c|cccc}
      x & 0 & 1 & 2 & 3 \\
      \hline	
      a(x) & 0 & 0 & 0 & 3 \\
      b(x) & 0 & 0 & 1 & 3 \\
      c(x) & 0 & 0 & 2 & 3 \\
      d(x) & 0 & 0 & 3 & 3 \\
      e(x) & 0 & 1 & 2 & 3 \\
      f(x) & 0 & 1 & 3 & 3 \\
      g(x) & 0 & 2 & 3 & 3 \\
      h(x) & 0 & 3 & 3 & 3
    \end{array}\\
    \\
  \end{array}
  \qquad
  \begin{array}{c}
    \\
    R_{8,2} \smallskip \\
    \begin{array}{c|cccc}
      x & 0 & 1 & 2 & 3 \\
      \hline	
      a(x) & 0 & 0 & 0 & 3 \\
      b(x) & 0 & 0 & 3 & 3 \\
      c(x) & 0 & 1 & 1 & 3 \\
      d(x) & 0 & 1 & 2 & 3 \\
      e(x) & 0 & 1 & 3 & 3 \\
      f(x) & 0 & 2 & 2 & 3 \\
      g(x) & 0 & 2 & 3 & 3 \\
      h(x) & 0 & 3 & 3 & 3
    \end{array}\\
    \\
  \end{array}
  \qquad
  \begin{array}{c}
    R_{10} \smallskip \\
    \begin{array}{c|cccc}
      x & 0 & 1 & 2 & 3 \\
      \hline	
      a(x) & 0 & 0 & 0 & 3 \\
      b(x) & 0 & 0 & 1 & 3 \\
      c(x) & 0 & 0 & 2 & 3 \\
      d(x) & 0 & 0 & 3 & 3 \\
      e(x) & 0 & 1 & 1 & 3 \\
      f(x) & 0 & 1 & 2 & 3 \\
      g(x) & 0 & 1 & 3 & 3 \\
      h(x) & 0 & 2 & 2 & 3 \\
      i(x) & 0 & 2 & 3 & 3 \\
      j(x) & 0 & 3 & 3 & 3
    \end{array}
  \end{array}
\end{gather*}
For space reasons we just show the addition and multiplication table
for the semiring $(R_{7,1},\vee,\circ)$:
\begin{displaymath}
  \begin{array}{c|ccccccc}
    \vee & a & b & c & d & e & f & g \\
    \hline
    a &  a & b & c & d & e & f & g \\
    b &  b & b & c & d & e & f & g \\
    c &  c & c & c & d & e & f & g \\
    d &  d & d & d & d & e & f & g \\
    e &  e & e & e & e & e & f & g \\
    f &  f & f & f & f & f & f & g \\
    g &  g & g & g & g & g & g & g
  \end{array}
  \qquad
  \begin{array}{c|ccccccc}
    \circ & a & b & c & d & e & f & g \\
    \hline
    a &  a & a & a & d & d & d & g \\
    b &  a & a & b & d & d & e & g \\
    c &  a & a & c & d & d & f & g \\
    d &  a & a & d & d & d & g & g \\
    e &  a & b & d & d & e & g & g \\
    f &  a & c & d & d & f & g & g \\
    g &  a & d & d & d & g & g & g
  \end{array}
\end{displaymath}

\subsection{$\infty_R$ is left but not right absorbing}

The following semiring is the unique finite simple additively
idempotent semiring with left but not right absorbing greatest
element, induced by ${(\{0,1\},\leq)}$:

\begin{displaymath}
  \begin{array}{c|cc}
    x &  0 & 1 \\
    \hline	
    a(x) & 0 & 0 \\
    b(x) & 0 & 1 \\
    c(x) & 1 & 1
  \end{array}\qquad
  \begin{array}{c|ccc}
    \vee & a & b & c \\
    \hline
    a &  a & b & c \\
    b &  b & b & c \\
    c &  c & c & c
  \end{array}\qquad
  \begin{array}{c|ccc}
    \circ & a & b & c \\
    \hline
    a &  a & a & a \\
    b &  a & b & c \\
    c &  c & c & c
  \end{array}
\end{displaymath}

\subsection{$\infty_R$ is absorbing}

The following semirings are all finite simple additively idempotent
semirings with absorbing greatest element, induced by
${(\{0,1,2\},\leq)}$:
\begin{displaymath}
  \begin{array}{c|ccc}
    x &  0 & 1 & 2 \\
    \hline	
    a(x) & 0 & 0 & 2 \\
    b(x) & 0 & 2 & 2 \\
    c(x) & 1 & 1 & 2 \\
    e(x) & 1 & 2 & 2 \\
    f(x) & 2 & 2 & 2
  \end{array}
  \qquad
  \begin{array}{c|ccccc}
    \vee  & a & b & c & d & e \\
    \hline
    a &  a & b & c & d & e \\
    b &  b & b & d & d & e \\
    c &  c & d & c & d & e \\
    d &  d & d & d & d & e \\
    e &  e & e & e & e & e \\
  \end{array}
  \qquad
  \begin{array}{c|ccccc}
    \circ  & a & b & c & d & e \\
    \hline
    a &  a & b & a & b & e \\
    b &  a & b & e & e & e \\
    c &  c & d & c & d & e \\
    d &  c & d & e & e & e \\
    e &  e & e & e & e & e \\
  \end{array}
\end{displaymath}

\begin{displaymath}
  \begin{array}{c|ccc}
    x &  0 & 1 & 2 \\
    \hline	
    a(x) & 0 & 0 & 2 \\
    b(x) & 0 & 1 & 2 \\
    c(x) & 0 & 2 & 2 \\
    d(x) & 1 & 1 & 2 \\
    e(x) & 1 & 2 & 2 \\
    f(x) & 2 & 2 & 2
  \end{array}
  \qquad
  \begin{array}{c|cccccc}
    \vee  & a & b & c & d & e & f \\
    \hline
    a &  a & b & c & d & e & f \\
    b &  b & b & c & d & e & f \\
    c &  c & c & c & e & e & f \\
    d &  d & d & e & d & e & f \\
    e &  e & e & e & e & e & f \\
    f &  f & f & f & f & f & f
  \end{array}
  \qquad
  \begin{array}{c|cccccc}
    \circ  & a & b & c & d & e & f \\
    \hline
    a &  a & a & c & a & c & f \\
    b &  a & b & c & d & e & f \\
    c &  a & c & c & f & f & f \\
    d &  d & d & e & d & e & f \\
    e &  d & e & e & f & f & f \\
    f &  f & f & f & f & f & f
  \end{array}
\end{displaymath}


\section*{Acknowledgement}

The authors would like to thank the anonymous referee for many
valuable suggestions which have increased the quality of this work
considerably.



\begin{thebibliography}{cc}

\bibitem{qt-1} K.\,Al-Zoubi, T.\,Kepka, P.\,N\v{e}mec:
  \emph{Quasitrivial Semimodules I}, Acta Univ.\ Carolin.\ Math.\
  Phys.~49 (2008), no.~1, 3--16
\bibitem{bashir_2001} R.\,El Bashir, J.\,Hurt, A.\,Jan\v{c}a\v{r}\'ik,
  T.\,Kepka: \textit{Simple Commutative Semirings}, Journal of Algebra
  236 (2001), 277--306
\bibitem{bashir} R.\,El Bashir, T.\,Kepka: \textit{Congruence-Simple
    Semirings}, Semigroup Forum 75 (2007), 588--608
\bibitem{birkhoff} G.\,Birkhoff: \textit{Lattice Theory}, 3rd ed.,
  Amer.\ Math.\ Soc., Providence, 1967
\bibitem{janowitz} T.\,S.\,Blyth, M.\,F.\,Janowitz: \emph{Residuation
    theory}, Pergamon Press, Oxford, 1972
\bibitem{howie} J.\,Howie: \textit{Fundamentals of Semigroup Theory},
  Oxford Univ.\ Press, Oxford, 1995
\bibitem{jezek} J.\,Je\v{z}ek, T.\,Kepka: \textit{The Semiring of
    1-Preserving Endomorphisms of a Semilattice}, Czechoslovak
  Mathematical Journal 59 (134), 2009, 999--1003
\bibitem{jezek2} J.\,Je\v{z}ek, T.\,Kepka, M.\,Mar\'oti: \textit{The
    endomorphism semiring of a semilattice}, Semigroup Forum 78
  (2009), 21--26
\bibitem{mitchell} S.\,S.\,Mitchell, P.\,B.\,Fenoglio:
  \textit{Congruence-free commutative semirings}, Semigroup Forum 37
  (1988), 79--91
\bibitem{monico} C.\,Monico: \textit{On finite congruence-simple
    semirings}, Journal of Algebra 271 (2004), 846--854
\bibitem{zum} J.\,Zumbr\"agel: \emph{Classification of finite
    congruence-simple semirings with zero}, Journal of Algebra and its
  Applications 7(3), 2008, 363--377

\end{thebibliography}
\end{document}